\documentclass[reqno,10pt]{amsart}
\usepackage[english]{babel}
\usepackage{graphicx}
\usepackage{amsmath}
\usepackage{amssymb}
\usepackage{amsthm}
\usepackage{libertine}
\usepackage[all]{xy}
\usepackage{amscd}

\usepackage{bm}
\usepackage{anysize}
\usepackage{graphicx}
\usepackage[hidelinks]{hyperref}
\usepackage{cleveref}
\usepackage{color}
\usepackage{fancybox}
\usepackage{bbm}
\usepackage{amsfonts}
\usepackage[T1]{fontenc}
\usepackage{latexsym}
\usepackage{cite}
\usepackage{mathrsfs} 

\usepackage{pst-node}
\usepackage{tikz-cd}
\usepackage{tikz}
\usepackage{float}

\usepackage{dsfont}

\usepackage{cleveref}

 \newtheorem{thm}{Theorem}[section]
 \newtheorem{cor}[thm]{Corollary}
 \newtheorem{lemma}[thm]{Lemma}
 \newtheorem{prop}[thm]{Proposition}
  
 \theoremstyle{definition}
 \newtheorem{definition}[thm]{Definition}
 \theoremstyle{remark}
 \newtheorem{rem}[thm]{Remark}
 \newtheorem*{ex}{Example}
 \numberwithin{equation}{section}
\usepackage{mathtools}
\DeclarePairedDelimiter\ceil{\lceil}{\rceil}
\DeclarePairedDelimiter\floor{\lfloor}{\rfloor}

\title[On the harmonic generalized Cauchy-Kovalevskaya extension and its connection with the Fueter-Sce theorem ]{On the harmonic generalized Cauchy-Kovalevskaya extension and its connection with the Fueter-Sce theorem}

\author[A. De Martino]{Antonino De Martino}
\address{(ADM)
	Politecnico di Milano\\Dipartimento di Matematica\\Via E. Bonardi, 9\\20133
	Milano, Italy
} \email{antonino.demartino@polimi.it}

\author[A.Guzmán]{ Ali Guzmán Adán }
\address{(AGA) (was Postdoc)
Clifford Research Group, Department of Electronics and Information Systems, Faculty of Engineering and Architecture, Ghent University, Krijgslaan 281, 9000 Gent, Belgium.} \email{ali.guzmanadan@ugent.be}

\begin{document}
\date{}

\maketitle

\begin{abstract}

One of the primary objectives of this paper is to establish a generalized Cauchy-Kovalevskaya extension for axially harmonic functions. We demonstrate that the result can be expressed as a power series involving Bessel-type functions of specific differential operators acting on two initial functions. Additionally, we analyze the decomposition of the harmonic CK extension in terms of integrals over the sphere $ \mathbb{S}^{m-1} $ involving functions of plane wave type.

Another key goal of this paper is to explore the relationship between the harmonic Cauchy-Kovalevskaya extension and the Fueter-Sce theorem. The Fueter-Sce theorem outlines a two-step process for constructing axially monogenic functions in $ \mathbb{R}^{m+1}$ starting from holomorphic functions in one complex variable. The first step generates the class of slice monogenic functions, while the second step produces axially monogenic functions by applying the pointwise differential operator $ \Delta_{\mathbb{R}{^{m+1}}}^{\frac{m-1}{2}} $ with $m$ being odd, known as the Fueter-Sce map, to a slice monogenic function.

By suitably factorizing the Fueter-Sce map, we introduce the set of axially harmonic functions, which serves as an intermediate class between slice monogenic and axially monogenic functions. In this paper, we establish a connection between the harmonic CK extension and the factorization of the Fueter-Sce map. This connection leads to a new notion of harmonic polynomials, which we show to form a basis for the Riesz potential. Finally, we also construct a basis for the space of axially harmonic functions.

\end{abstract}

\tableofcontents

\section{Introduction}
The classical Cauchy-Kovalevskaya extension provides a method for characterizing the solutions of certain partial differential equations (PDEs) through their restriction to a submanifold of codimension one. When applied to the Cauchy-Riemann equation, this principle demonstrates that a holomorphic function defined in a suitable region of the complex plane is entirely determined by its restriction to the real axis. Such an extension principle for holomorphic functions has been further generalized within the framework of Clifford analysis which is a higher dimensional generalization of complex analysis and a refinement of harmonic analysis, see \cite{BDS, green}. 
\\ The main topic of Clifford analysis is the study of monogenic functions, i.e. null solutions of the generalized Cauchy-Riemann operator $\mathcal{D}:=\partial_{x_0}+ \partial_{\underline{x}}$, where $\partial_{\underline{x}}:= \sum_{j=1}^m e_j \partial_{x_j}$, in an open set of $ \mathbb{R}^{m+1}$. We denote by $(e_1,...,e_m)$ an orthogonal basis in $ \mathbb{R}^m$ and $ \underline{x}=\sum_{j=1}^m x_j e_j$ is a vector variable defined in $\mathbb{R}^m$. Any monogenic function is determined by its restriction to the hyperplane $x_0=0$. On the other side any real-analytic function $ f( \underline{x})$ defined in a region of $ \mathbb{R}^m$ has a unique monogenic extension $f(x_0, \underline{x})$ called Cauchy-Kovalevskaya (CK) extension, see Theorem \ref{GCK1}.
\\ The first aim of this paper is to study a generalized CK-extension that can generate harmonic functions. In this case we get a CK-extension that depends on two initial functions on which power series of differential operator act. We also express the generalized CK-extension for harmonic functions using integrals over the $(m-1)$-dimensional sphere, involving 1-vectors of plane wave-type functions.
 In Clifford analysis the harmonic functions appear also in relation with the Fueter-Sce theorem

\medskip

In literature a constructive way to build monogenic functions in $\mathbb{R}^4$ is the Fueter theorem, see \cite{F}. This result asserts that every intrinsic holomorphic function induces a  monogenic function. A holomorphic function $f(z)= \alpha(u,v)+i \beta(u,v)$ (with $z=u+iv$) is said to be intrinsic if it is defined on complex domains symmetric with respect to the real axis, and its restriction to the real line is real-valued, see Definition \ref{holoint}.
\\ More than 20 years later M. Sce extended the Fueter's result to $ \mathbb{R}^{m+1}$ for odd values of the dimension $m$, see \cite{S, CSS2}. In this case, given a holomorphic complex function $f$ as above, the Clifford-valued function
$$x \mapsto \Delta_{\mathbb{R}^{m+1}}^{\frac{m-1}{2}} \left( \alpha(x_0, |\underline{x}|)+ \frac{\underline{x}}{|\underline{x}|}\beta(x_0, |\underline{x}|)\right)$$
is in the kernel of the generalized Cauchy-Riemann operator $ \partial_{x_0}+ \partial_{\underline{x}}$. Here $\Delta_{\mathbb{R}^{m+1}}= \sum_{j=0}^m \partial^2_{x_j}$ is the Laplace operator in $m+1$ real variables. We observe that, since $m$ is an odd number, in this case the operator $\Delta_{\mathbb{R}^{m+1}}^{\frac{m-1}{2}}$ is a pointwise differential operator.
\\ In 1997 T. Qian extended Sce's results to all dimensions by using the Fourier multiplier to define the fractional Laplacian, see \cite{Q, TAOBOOK}. This leads to the so well known Fueter-Sce-Qian Theorem, which is a constructive way to build monogenic functions in $\mathbb{R}^{m+1}$ starting from intrinsic holomorphic functions. This link is established by means of the Fueter-Sce-Qian map $\tau_m$ which coincides with $ \Delta_{\mathbb{R}{^{m+1}}}^{\frac{m-1}{2}}$ when $m$ is odd.  
\\ We remark that the range of the Fueter-Sce-Qian map, as well as the range of the generalized CK-extension, is the module of monogenic functions of axial-type. These are monogenic functions of the form $A(x_0,|\underline{x}|)+ \frac{\underline{x}}{|\underline{x}|} B(x_0,|\underline{x}|)$ where the functions $A$ and $B$ are Clifford-valued. We shall denote the module of axial monogenic functions in $ \mathbb{R}^{m+1}$ by $ \mathcal{AM}(\mathbb{R}^{m+1})$, see Definition \ref{axmono}.
\\ It is important to observe that the Fueter-Sce-Qian extension is provided in two steps. The first one turns the holomorphic complex function $f$ into a Clifford-valued function defined in $ \mathbb{R}^{m+1}$ by means of the following map
\begin{equation}
\label{introslice}
\alpha(u,v)+i \beta(u,v) \mapsto \alpha(x_0, |\underline{x}|)+ \frac{\underline{x}}{|\underline{x}|} \beta(x_0, |\underline{x}|).
\end{equation}
This map is called the slice extension and it maps intrinsic holomorphic functions in $ \mathbb{C}$ to slice monogenic functions in $ \mathbb{R}^{m+1}$, which is a Clifford-module denoted by $ \mathcal{SM}(\mathbb{R}^{m+1})$, see Definition \ref{slicemono}. Due to the equivalence between the space of intrinsic holomorphic functions in $ \mathbb{C}$, and the space of $\mathbb{R}$-valued real-analytic functions on $ \mathbb{R}$, denoted by $ \mathcal{A}(\mathbb{R})$, the map \eqref{introslice} can be naturally defined on real-valued functions of one real variable. In this way we obtain a map $S: \mathcal{A}(\mathbb{R}) \otimes \mathbb{R}_m \to \mathcal{SM}(\mathbb{R}^{m+1})$ which factorizes the Fueter-Sce-Qian extension as follows 
\begin{equation}
\label{diagintro}
\mathcal{A}(\mathbb{R}) \otimes \mathbb{R}_m\overset{S}{\longrightarrow} \mathcal{SM}(\mathbb{R}^{m+1})\overset{\tau_m}{\longrightarrow} \mathcal{AM}(\mathbb{R}^{m+1}).
\end{equation}
We refer to the papers \cite{Elb, KQS, D, PQS,Somm} for possible generalization of the Fueter-Sce-Qian mapping theorem.
In \cite{DDG} by combining \eqref{diagintro} and the isomorphism of the generalized CK-extension for monogenic functions
\begin{equation*}
GCK: \mathcal{A}(\mathbb{R}) \otimes \mathbb{R}_m \to \mathcal{AM}(\mathbb{R}^{m+1}),
\end{equation*}  
we get the following relation between Fueter-Sce theorem and the generalized CK-extension:
$$ \tau_m \circ S= \gamma_m GCK \circ \partial_{x_0}^{m-1}.$$
This gives rise to the following commutative diagram

\begin{equation}
\label{diagintro3}
 \begin{tikzcd}[row sep = 3em, column sep = 6em]
 \mathcal{A}(\mathbb{R}) \otimes \mathbb{R}_m \arrow[r, "S", rightarrow] \arrow[d, "\gamma_m \partial_{x_0}^{m-1}", labels=left] & \mathcal{SM}(\mathbb{R}^{m+1}) \arrow[d, "\Delta^{\frac{m-1}{2}}_{\mathbb{R}^{m+1}}"] \\
 		\mathcal{A}(\mathbb{R}) \otimes \mathbb{R}_m \arrow[r, "\textup{GCK}", rightarrow]                  &\mathcal{AM}(\mathbb{R}^{m+1})
 	\end{tikzcd}
 \end{equation}

In \cite{ANST}, it is shown that by factorizing the Fueter-Sce map, represented by the differential operator $ \Delta_{\mathbb{R}^{m+1}}^{\frac{m-1}{2}}$, the set of axially harmonic functions occupies an intermediate position between the set of slice monogenic functions and the set of axially monogenic functions within the Fueter-Sce framework.
 Indeed, we can write the following factorization of the diagram \eqref{diagintro}
\begin{equation}
\label{introdiag2}
\begin{CD}
\textcolor{black}{\mathcal{A}(\mathbb{R}) \otimes \mathbb{R}_m}  @>S>> \textcolor{black}{\mathcal{SM}(\mathbb{R}^{m+1})}  @>\ \  \Delta_{\mathbb{R}^{m+1}}^{\frac{m-3}{2}} \mathcal{D} >>\textcolor{black}{\mathcal{AH}(\mathbb{R}^{m+1})}@>\ \  \mathcal{\overline{D}} >>\textcolor{black}{\mathcal{AM}(\mathbb{R}^{m+1})},
\end{CD}
\end{equation}
where $ \mathcal{\overline{D}}:= \partial_{x_0}- \partial_{\underline{x}}$.
Here $\mathcal{AH}(\mathbb{R}^{m+1})$ is the set of axially harmonic functions, i.e. functions in the kernel of $ \Delta_{\mathbb{R}^{m+1}}$ and of axial type, see Definition \ref{aah}.
\\ By combining \eqref{introdiag2} with the harmonic CK-extension 
$$ HGCK: ( \mathcal{A}(\mathbb{R}) \otimes \mathbb{R}_m)^2 \to \mathcal{AH}(\mathbb{R}^{m+1}),$$
we aim to provide a factorization of the diagram \eqref{diagintro3}.\\
\\ As a direct application of our results, we derive an explicit expression for the action of the operator $ \Delta_{\mathbb{R}{^{m+1}}}^{\frac{m-3}{2}} \mathcal{D}$ when applied to the monomials $(x_0 + \underline{x})^k $, where $ k \in \mathbb{N}$. This leads to the introduction of a new class of harmonic polynomials, defined as follows:  
\begin{equation}
\label{hh}
\mathcal{P}_k^m(x) := \sum_{s=0}^k T_{s}^k(m) x^{k-s} \bar{x}^s, \qquad T_s^k(m) := \binom{k}{s} \frac{\left(\frac{m-1}{2}\right)_{k-s} \left(\frac{m-1}{2}\right)_{s}}{(m-1)_k},
\end{equation}
where $ \overline{x} := x_0 - \underline{x} $.  

Furthermore, we establish that these polynomials are connected to the Chebyshev polynomials and form a basis for the Riesz potential $|1-x|^{-(m-1)} $.
\\ Our final goal is to establish a basis for the set of axially harmonic functions. In the monogenic setting the application of the Fueter-Sce-Qian map to the monomial $(x_0+ \underline{x})^k$ gives a well-known family of monogenic Clifford-Appell polynomials. These polynomials form a basis for the axially monogenic functions.
In the harmonic setting the polynomials $ \mathcal{P}_k^m(x)$ are not a basis for the set of axially harmonic functions. We build a basis by considering harmonic polynomials that comes from the scalar and the 1-vector parts of the monogenic Clifford-Appell polynomials. Finally, we study some Appell-like properties for the polynomials which are basis of the axially harmonic functions.

\medskip

Summarizing in this paper we solve the following problems
\begin{itemize}
\item[\textbf{P1}] establish a generalized CK-extension for axially harmonic functions
\item[\textbf{P2}] obtain a connection between the generalized CK-extension for axially harmonic functions and the Fueter-Sce theorem in order to give a factorization of the diagram \eqref{diagintro3}.
\item[\textbf{P3}] Construct a basis for the set of axially harmonic functions.
\end{itemize}

\emph{Outline of the paper: }
The plan of the paper is as follows. In Section 2 we recall briefly the main notions of Clifford analysis and we also establish some minor results that will be useful in subsequent sections. We also revise how to connect the generalized CK-extension and the Fueter-Sce theorem. In Section 3 we give an answer to \textbf{P1}, by developing a generalized CK-extension for axially harmonic functions. In Section 4 we show how to decompose the generalized CK-extension in terms of integrals of plane waves over the sphere of 1-vectors in $ \mathbb{R}^{m+1}$.  Section 5 is devoted to study a connection between the generalized CK-extension and a factorization of the Fueter-Sce map, this provides a solution to \textbf{P2}. Finally, in Section 6 we introduce the harmonic polynomials \eqref{hh} and we provide a connection of these polynomials with a factorization of the Fueter-Sce map and the generalized CK-extension for axially harmonic functions. In this section we also provide a basis for the set of axially harmonic functions, this solves \textbf{P3}.

\section{Preliminaries}
In this section we fix the notations and we briefly revise some results from Clifford analysis.
\\ Let $ \mathbb{R}_m$ be the real associative Clifford algebra generated by $e_1,...,e_m$ satisfying the relations $e_je_{\ell}+e_{\ell}e_j=-2 \delta_{j \ell}$, for $j$, $ \ell=1,...,m$ where $ \delta_{j \ell}$ is the Kronecker symbol. An element in the Clifford algebra is denoted by 
$$ \sum_{A\subset M} e_A x_A, \qquad x_A \in \mathbb{R},$$
where $A:= \{1,...,m\}$ for any multi-index $M= \{\ell_1,..., \ell_k\}$ with $ \ell_1<...<\ell_k$. We set $e_A:= e_{\ell_1}...e_{\ell_k}$, $e_{\emptyset}=1$ and $|A|=k$. We can write every $x \in \mathbb{R}_m$ as a multivector decomposition
$$ x=\sum_{k=0}^m [x]_k, \qquad \hbox{where} \qquad [x]_k= \sum_{|A|=k} a_A e_A.$$
We denote by $[.]_k: \mathbb{R}_m \to \mathbb{R}_m^{(k)}$ the canonical projection of $ \mathbb{R}_m$ onto the space of $k$-vectors $ \mathbb{R}_m^{(k)}= \hbox{span}_{\mathbb{R}}\{e_A\, : \, |A|=k\}$. We observe that $ \mathbb{R}_m^{(0)}= \mathbb{R}$, whereas the space of 1-vectors $ \mathbb{R}_m^{(1)}$ is isomorphic to $ \mathbb{R}^m$. We note also that $ \mathbb{R}_1 \simeq \mathbb{C}$ and $ \mathbb{R}_2 \simeq \mathbb{H}$, where $\mathbb{H}$ is the set of quaternions.
\\ An important subspace of the real Clifford algebra $ \mathbb{R}_{m}$ is $ \mathbb{R}_{m}^{0} \oplus \mathbb{R}_m^{(1)}$ whose elements are called paravectors. This space can be identified with $ \mathbb{R}^{m+1}$ and we denote its elements by
$$ x=x_0+ \underline{x}= x_0+ \sum_{j=1}^{m} x_j e_j,$$ 
where $ x_0 \in \mathbb{R}$ and $ \underline{x}= \sum_{j=1}^m x_j e_j \in \mathbb{R}^m$. The conjugate of $x$ is denoted by $ \bar{x}=x_0- \underline{x}$ and the Eucldian modulus is given by $ |x|^2=x_0^2+x_1^2+...+x_m^2$. Similarly, for any pair of vectors $ \underline{x}$, $\underline{y} \in \mathbb{R}^m$ one has that
$$ \underline{x} \underline{y}+ \underline{y} \underline{x}=-2 \sum_{j=1}^m x_j y_j=-2 \langle \underline{x}, \underline{y} \rangle,$$
where $\langle \underline{x}, \underline{y} \rangle$ denotes the Euclidean inner product in $ \mathbb{R}^m$, and so, $ \underline{x}^2=-|\underline{x}|^2$ where $| \underline{x}|^2:= \sum_{j=1}^m x_j^2$ is the square of the Euclidean norm of $ \underline{x}$ in $ \mathbb{R}^m$. We shall also make use of the \emph{wedge product} for vectors.
 $\underline{x}$, $\underline{y} \in \mathbb{R}^m$, that is defined by
$$ \underline{x} \wedge \underline{y}= \sum_{j <k}e_j e_k(x_j y_k-x_ky_j)= \frac{1}{2}(\underline{x} \underline{y}- \underline{y} \underline{x}).$$
\medskip

Let $E$ be an open subset of $ \mathbb{R}^m$. In this paper we will mainly consider real-analytic functions with values on the Clifford algebra $ \mathbb{R}_m$. We denote this set of functions by $ \mathcal{A}(E) \otimes \mathbb{R}_m$. Within this module we wil take into account two possible expansions of holomorphic functions of one complex variable to higher dimensions: the \emph{slice monogenic functions} and \emph{axially monogenic functions}.
\medskip 

We start by recalling the definition of slice monogenic function, see \cite{CSS3, CSS4}. Let us denote by $ \mathbb{S}^{m-1}$ the $(m-1)$-sphere of unit 1-vectors in $ \mathbb{R}^{m+1}$, i.e.
$$ \mathbb{S}^{m-1}= \{\underline{x}=x_1e_1+...+x_me_m\, : \, | \underline{x}|=1 \}.$$
Given $ \underline{\omega} \in \mathbb{S}^{m-1}$ we denote the 2-dimensional real subspace of $ \mathbb{R}^{m+1}$ passing through $1$ and $ \underline{\omega} \in \mathbb{S}^{m-1}$ as
$$ \mathbb{C}_{\underline{\omega}}= \{u+v  \underline{\omega}\, : \, u,v \in \mathbb{R}\}.$$
This is an isomorphic copy of the complex plane.

\begin{definition}
\label{slicemono}
Let $ \Omega \subset \mathbb{R}^{m+1}$ be an open set. A function $f: \Omega \to \mathbb{R}_m$ is said to be (left) slice monogenic for any $ \underline{\omega} \in \mathbb{S}^{m-1}$, if the restrictions of $f$ to the complex planes $ \mathbb{C}_{\underline{\omega}}$ are holomorphic on $ \Omega \cap \mathbb{C}_{\underline{\omega}}$, i.e.
$$ (\partial_u+ \underline{\omega}\partial_v) f(u+ \underline{\omega}v)=0.$$
The right $ \mathbb{R}_m$-module of slice monogenic functions on $ \Omega$ is denoted by $ \mathcal{SM}(\Omega)$.
\end{definition}
A slice monogenic function has good properties on suitable sets called \emph{axially slice domains}. Let us recall the definition.
\begin{definition}
\label{sets}
An open set $ \Omega$ of $ \mathbb{R}^{m+1}$ is said to be an \emph{axially symmetric domain} if it satisfies the following conditions:
\begin{itemize}
\item $ \Omega \cap \mathbb{R} \neq \emptyset$,
\item $\Omega \cap \mathbb{C}_{\underline{\omega}}$ is a domain (i.e. an open and connected set) in $ \mathbb{C}_{\underline{\omega}}$ for all $ \underline{\omega} \in \mathbb{S}^{m-1}$
\item $ \Omega$ is a $SO(m)$-invariant with respect to the real axis. This means that if $x_0+ \underline{x} \in \Omega$ then for all $M \in SO(m)$ we have that $x_0+ M \underline{x} \in \Omega$.
\end{itemize}
\end{definition}
We note that if an open set satisfies only the first two bullets of the above definitions is usually called \emph{slice domain}, whereas if the third point is satisfied the set is known as \emph{axially symmetric open set}.

We observe that if we consider $m=1$ in Definition \ref{sets} we reduce to the notion of open set in the complex plane symmetric to the real axis. In this particular case we say that $ \Omega$ is an \emph{intrinsic complex domain}.

Equivalently, one can define axially symmetric slice domains in terms of intrinsic complex domains. A set $\Omega$ in $ \mathbb{R}^{m+1}$ is an axially symmetric slice domain if and only if there exists an intrinsic complex domain $ \Omega_2 \subset \mathbb{C}$ such that
$$\Omega= \Omega_2 \times \mathbb{S}^{m-1}= \{u+ \underline{w}v \,: \, (u,v) \in \Omega_2, \, \underline{\omega} \in \mathbb{S}^{m-1}\}.$$ 

\medskip

In the rest of the paper we will always work with slice monogenic functions defined on axially symmetric slice domains. Since in such domains these functions satisfy good properties, as it is stated in the following representation theorem.

\begin{thm}[Representation theorem \cite{CSS3}]
\label{rapp}
Let $ \Omega \subset \mathbb{R}^{m+1}$ be a an axially symmetric slice domain, and let $ \Omega_2 \subset \mathbb{C}$ be the intrinsic complex domain such that $ \Omega= \Omega_2 \times \mathbb{S}^{m-1}$. Then any function $f \in \mathcal{SM}(\Omega)$ can be written as
$$ f(u+ \underline{\omega}v)= \alpha(u,v)+ \underline{\omega} \beta(u,v),$$
where the functions $ \alpha$, $ \beta$ are Clifford-valued real functions defined in $ \Omega_2$, such that
\begin{equation}
\label{cond}
\alpha(u,v)=\alpha(u,-v), \qquad \beta(u,v)=- \beta(u,-v) \qquad \hbox{for all} \quad (u,v) \in \Omega_2
\end{equation}
and moreover satisfy the Cauchy-Riemann system
\begin{equation}
\label{CR1}
\begin{cases}
	\partial_u \alpha(u,v)- \partial_v \beta(u,v)=0\\
	\partial_v \alpha(u,v)+ \partial_u \beta(u,v)=0.
\end{cases}
\end{equation}
\end{thm}

From the previous result one can get how to induce slice monogenic functions from intrinsic holomorphic functions.

\begin{definition}
\label{holoint}
Let us consider a holomorphic function $f(z)= \alpha(u,v)+i \beta(u,v)$, with $z=u+iv$. This function is said to be \emph{intrinsic} if it is defined in a complex domain $\Omega_2$ and satisfies the following relation $f(z)^c=f(z^c)$, where $.^c$ is the complex conjugation. Equivalently, this condition means that the real and imaginary part of $f$ satisfy the condition \eqref{cond}, or analogously, that the restriction of the function $f$ to the real line is real-valued.
\\ We denote by $Hol(\Omega_2) $ the space of holomorphic complex functions on $ \Omega_2 $ and by $\mathcal{H}(\Omega_2) $ the real vector subspace of $ Hol(\Omega_2) $ consisting of complex intrinsic holomorphic functions, i.e.,
\begin{eqnarray*}
\mathcal{H}(\Omega_2)&=& \{f \in Hol(\Omega_2) \,: \, \alpha(u,v)=\alpha(u,-v), \quad \beta(u,v)= -\beta(u,-v)\}\\
&=&\{f \in Hol(\Omega_2)\, : \, f_{| \mathbb{R}} \, \hbox{is} \, \mathbb{R}-\hbox{valued}\}.
\end{eqnarray*} 
\end{definition}
The set of intrinsic holomorphic functions defined in $ \Omega_2 \subset \mathbb{C}$ induces the class of slice monogenic functions on $\Omega= \Omega_2 \times \mathbb{S}^{m-1}$ by means of the following extension map 
$$S_{\mathbb{C}}: \mathcal{H}(\Omega_2) \otimes \mathbb{R}_m \to \mathcal{SM}(\Omega), \qquad \alpha(u,v)+i \beta(u,v) \to \alpha(x_0, |\underline{x}|)+ \underline{\omega} \beta(x_0, |\underline{x}|),$$
which consists of replacing the complex variable $z=u+iv$ by the paravector variable $x=x_0+ \underline{x}$, and the complex unit $i$ is replaced by the uni vector $ \underline{\omega}= \frac{\underline{x}}{|\underline{x}|}$ in $ \mathbb{R}^m$.  
The slice monogenic function obtained from the extension maps is of the form $f(x_0+ \underline{x})=\alpha(x_0,|\underline{x}|)+ \underline{\omega} \beta(x_0, |\underline{x}|)$. This function is well defined at $\underline{x}=0$, since $ \beta(x_0,|\underline{x}|)$ is an odd function in the second variable. Indeed, by using the Taylor expansion of the function $ \beta$
$$ \beta(x_0,|\underline{x}|)= \sum_{j=0}^\infty \frac{|\underline{x}|^{2j+1}}{(2j+1)!} \partial_{|\underline{x}|}^{2j+1} [\beta](x_0,0),$$
this implies that
$$ \frac{\underline{x}}{|\underline{x}|} \beta(x_0,|\underline{x}|)\biggl |_{\underline{x}=0}=\sum_{j=0}^\infty \frac{\underline{x}|\underline{x}|^{2j}}{(2j+1)!}\biggl |_{\underline{x}=0} \partial_{|\underline{x}|}^{2j+1} [\beta](x_0,0)=0.$$

Since the intrinsic holomorphic functions are uniquely determined by their restrictions to the real line, the above extension can be realized in terms of real analytic functions on the real line. Let $f(u+iv)= \alpha(u,v)+ i \beta(u,v)$ be an intrinsic holomorphic function on $\Omega_2$. The Taylor expansion of the function $f$ around a point $(u,0)$ is
$$ f(u+iv)=\sum_{j=0}^\infty \frac{(iv)^j}{j!} f^{(j)}(u)= \sum_{j=0}^\infty \frac{(iv)^j}{j!} \partial_u^j[\alpha](u,0).$$
It is clear that the function $f$ is the unique holomorphic extension of the function $f_0(u)=\alpha(u,0)$. Furthermore, we observe that the real and imaginary parts of $f$ are given by
$$ \alpha(u,v)= \sum_{j=0}^\infty \frac{(-1)^j v^{2j}}{(2j)!} \partial_u^{2j}[\alpha](u,0), \qquad \hbox{and} \qquad \beta(u,v)=\sum_{j=0}^\infty  \frac{(-1)^j v^{2j+1}}{(2j+1)!} \partial_u^{2j+1}[\alpha](u,0).
$$
Now, let us consider the subset of the real line $\Omega_1=\Omega_2 \cap \mathbb{R}$, where $\Omega_2$ is an intrinsic complex domain. We denote by $ \mathcal{A}(\Omega_1)$ the space of real-valued analytic functions on $\Omega_1$ with unique holomorphic extension to $ \Omega_2$. We can define the holomorphic extension map $C=exp(iv \partial_u)$ i.e.
$$ C: \mathcal{A}(\Omega_1) \to \mathcal{H}(\Omega_2), \qquad f_0(u) \mapsto \sum_{j=0}^\infty \frac{(iv)^{j}}{j!} f_{0}^{(j)}(u).$$    
By means of the above map we can define the slice monogenic extension map as $S= S_{\mathbb{C}} \circ C= exp(\underline{x} \partial_{x_0})$, i.e. 
\begin{equation}
\label{26}
S: \mathcal{A}(\Omega_1) \otimes \mathbb{R}_m \to \mathcal{SM}(\Omega), \qquad f_0(x_0) \mapsto \sum_{j=0}^\infty \frac{\underline{x}^j}{j!} f_0^{(j)}(x_0).
\end{equation}
To sum up we have the following result.
\begin{thm}[Slice extension map]
\label{sliceex}
Under the assumptions stated above, the following $ \mathbb{R}_m$ modules are isomorphic,
$$ \mathcal{SM}(\Omega) \simeq \mathcal{A}(\Omega_1) \otimes \mathbb{R}_m \simeq \mathcal{H}(\Omega_2) \otimes \mathbb{R}_m.$$
In fact the following diagram  commutes
\[
\begin{tikzcd}[row sep = 2em, column sep = 5em]
\mathcal{A}(\Omega_1)\otimes \mathbb{R}_m \arrow[r, "C", ] \arrow[dr, "S", labels=below] & \mathcal{H}(\Omega_2) \otimes \mathbb{R}_m \arrow[d, "S_{\mathbb{C}}"] \\
& \mathcal{SM}(\Omega),
\end{tikzcd}
\]
where the map $S=S_{\mathbb{C}} \circ C= exp(\underline{x}\partial_{x_0}$), given by $S[f_0](x)=\sum_{j=0}^\infty \frac{\underline{x}^j}{j!} f_0^{(j)}(x_0)$ is inverted by the restriction operator to the real line, i.e. $S[f_0]|_{\underline{x}=0}=f_0.$
\end{thm}

Another possibility to extend the holomorphic function theory of one compelex variable to higher dimensions is given by the theory of monogenic functions. Standard references on this setting are \cite{BDS, green}. 

\begin{definition}
Let $ \Omega$ be an open set in $ \mathbb{R}^{m+1}$. A function $f: \Omega \subset \mathbb{R}^{m+1} \to \mathbb{R}_m$, is called (left) monogenic if it is in the kernel of the generalized Cauchy-Riemann operator in $ \mathbb{R}^{m+1}$ i.e. 
$$ \mathcal{D}f=( \partial_{x_0}+ \partial_{\underline{x}})f(x)=0,$$
where $ \partial_{\underline{x}}=\sum_{j=1}^m e_j \partial_{x_j}$ is the well-known Dirac operator in $ \mathbb{R}^m$.
\end{definition}

\begin{rem}
We observe that if $m=3$ the operator $\mathcal{D}$ is called Fueter operator.
\end{rem}

We observe that for $m=1$ , the monogenic functions on $ \mathbb{R}^2$ are holomorphic functions of the variable $x_0+e_1x_1$. The theory of monogenic functions can be considered a refinement of the harmonic analysis, since the Cauchy-Riemann operator $ \mathcal{D}$ factorizes the Laplace operator in $ \mathbb{R}^{m+1}$, i.e.
$$ \Delta_{\mathbb{R}^{m+1}}:=\partial_{x_0}^2+ \Delta_{\underline{x}}= \mathcal{D} \overline{\mathcal{D}}= \overline{\mathcal{D}} \mathcal{D},$$
where $ \Delta_{\underline{x}}:= \sum_{j=1}^m \partial_{x_j}^2$ and $\overline{\mathcal{D}}:= \partial_{x_0}-\partial_{\underline{x}}$.
The action of the of the Laplace operator in $ \mathbb{R}^m$ on integer powers of $ \underline{x} \in \mathbb{R}^m$ will be fundamental in the next section.
\begin{lemma}
Let $j  \geq 2$. Then for $\underline{x} \in \mathbb{R}^{m}$ we have that
\begin{equation}
	\label{actla}
	\Delta_{\underline{x}} (\underline{x}^{j})= c(m,j) \underline{x}^{j-2}, \quad \hbox{where} \quad c(m,j)= \begin{cases}
		-4p \left( \frac{m}{2}+p-1\right)& \quad \hbox{if} \, j=2p, \quad p \in \mathbb{N} \\
		-4p \left( \frac{m}{2}+p\right)& \quad \hbox{if} \, j=2p+1, \quad p \in \mathbb{N}.
	\end{cases}
\end{equation}
\end{lemma}
\begin{proof}
For $\ell \geq 1, k \geq 0$, it is crucial to use the following formula, see \cite{green}: 
\begin{equation}
\label{lap}
\Delta_{\underline{x}}(| \underline{x}|^{2\ell}H_k(\underline{x}))=4\ell \left(k+\frac{m}{2}+\ell-1\right)|\underline{x}|^{2(\ell-1)}H_k(\underline{x}), \qquad  \underline{x} \in \mathbb{R}^m,
\end{equation}
where $H_k(\underline{x})$ are spherical harmonics of degree $k$. Now, we suppose that $j$ is even. Thus if $j=2p$, with $p \in \mathbb{N}$, by formula \eqref{lap} (with $k=0$), we have
$$ \Delta_{\underline{x}}(\underline{x}^{2p})=(-1)^p \Delta_{\underline{x}}(| \underline{x}|^{2p})=(-1)^p4p \left(\frac{m}{2}+p-1\right) | \underline{x}|^{2(p-1)}=-4p \left(\frac{m}{2}+p-1\right) \underline{x}^{2p-2}.$$
Now, we suppose that $j$ is odd. So we consider $j=2p+1$, with $p \in \mathbb{N}$, so by \eqref{lap} (with $k=1$) we obtain
$$ \Delta_{\underline{x}}(\underline{x}^{2p+1})=(-1)^p \Delta_{\underline{x}}(| \underline{x}|^{2p} \underline{x})=(-1)^p 4p \left(\frac{m}{2}+p\right)| \underline{x}|^{2(p-1)} \underline{x}=- 4p \left(\frac{m}{2}+p\right) \underline{x}^{2p-1}.$$
This proves the result.
\end{proof}

An interesting subset of monogenic functions is given by the following class of functions.

\begin{definition}[Axially monogenic function]
\label{axmono}
Let $\Omega \subset \mathbb{R}^{m+1}$ be an axially symmetric slice domain and let $\Omega_2 \subset \mathbb{C}$ be an intrinsic complex domain such that $\Omega=\Omega_2 \times \mathbb{S}^{m-1}$. A function $f \in \mathcal{A}(\Omega) \otimes \mathbb{R}_m$ is said to be  axially monogenic if it is in the kernel of the generalized Cauchy-Riemann operator in $ \mathbb{R}^{m+1}$, i.e. $ \mathcal{D}f=0$, and it is of the form 
\begin{equation}
\label{form}
f(x_0+ \underline{x})=\alpha(x_0, |\underline{x}|)+ \underline{\omega} \beta(x_0,|\underline{x}|), \qquad \underline{\omega}= \frac{\underline{x}}{|\underline{x}|},
\end{equation}
where $\alpha$, $\beta \in \mathcal{A}(\Omega_2) \otimes \mathbb{R}_m$ satisfy the conditions \eqref{cond}. We denote by $ \mathcal{AM}(\Omega)$ the set of left axially monogenic functions on $ \Omega$.
\end{definition}

The axially monogenic monogeinc functions satisfy the so called Vekua system, see \cite{green}.

\begin{thm}
Let $m \in \mathbb{N}$. Let $f$ be an axially monogenic function as in Definition \ref{axmono}. Then the functions $\alpha(x_0,r)$, $\beta(x_0,r)$ satisfy the Vekua system:
\begin{equation}
\label{vek}
\begin{cases}
\partial_{x_0}\alpha(x_0,r)-\partial_r \beta(x_0,r)= \frac{m-1}{r}\beta(x_0,r)\\
\partial_{x_0}\beta(x_0,r)+\partial_r \alpha(x_0,r)=0,
\end{cases}
\end{equation}
with $r=| \underline{x}|$.
\end{thm}

\begin{rem}

	From the fact that the functions $A$ and $B$ satisfy the condition \eqref{cond} we have that $f(x_0+ \underline{x})$ is well-defined at $ \underline{x}=0$. By using the Taylor expansion of $A$ and $B$ around points in the real line 
	$$ \alpha(x_0, |\underline{x}|)= \sum_{j=0}^\infty |\underline{x}|^{2j} (-1)^j f_{2j}(x_0) \qquad \hbox{and} \qquad \beta(x_0, |\underline{x}|)= \sum_{j=0}^\infty |\underline{x}|^{2j+1} (-1)^j f_{2j+1}(x_0).$$
	We can write the function $f$ in power series in terms of $ \underline{x}$, i.e.
	$$ f(x)= \sum_{j=0}^\infty \underline{x}^{j} f_{j}(x_0).$$
\end{rem}
The action of the generalized Cauchy-Riemann operator, its conjugate, and the Laplace operator on a function of axial form is described by the following result; see \cite{D, green, GHS}.
\begin{lemma}
\label{Newt}
Let $m \in \mathbb{N}$. Let be a function  $f$ of axial type, i.e. of the form 
$$ f(x_0+ \underline{x})=\alpha(x_0, r)+\underline{\omega}\beta(x_0, r), \qquad \underline{\omega}= \frac{\underline{x}}{| \underline{x}|}, \quad r=| \underline{x}|.$$
Then the application of the operator $ \mathcal{D}$ on the function $f$ can be written as
\begin{equation}
\mathcal{D}f(x_0+ \underline{x})=\left( \partial_{x_0}\alpha(x_0,r)- \partial_r\beta(x_0,r)- \frac{m-1}{r} \beta(x_0,r)\right)+ \underline{\omega} (\partial_{x_0}\beta(x_0,r)+ \partial_{r}\alpha(x_0,r))\label{Dirac}
\end{equation}
and the application of the operator $\overline{\mathcal{D}}$ on the function $f$ is given by
$$ 
\mathcal{\overline{D}}f(x_0+ \underline{x})=\left( \partial_{x_0}\alpha(x_0,r)+ \partial_r\beta(x_0,r)+ \frac{m-1}{r} \beta(x_0,r)\right)+ \underline{\omega} (\partial_{x_0}\beta(x_0,r)- \partial_{r}\alpha(x_0,r)). 
$$
Moreover the application of the Laplace operator can be written as
\begin{eqnarray}
\label{lapp}
\Delta f(x_0+\underline{x})&=& \partial_{x_0}^2\alpha(x_0,r)+\partial_{x_0}^2\beta(x_0,r)+\partial_{r}^2\alpha(x_0,r)+ \underline{\omega}\partial_{r}^2\beta(x_0,r)\\
\nonumber
&&+ \frac{2}{r} \partial_r \alpha(x_0,r)+ (m-1) \underline{\omega} \partial_r \left( \frac{\beta(x_0,r)}{r}\right).
\end{eqnarray}
\end{lemma}

As it happens for the slice monogenic functions, there exist an isomorphism between the module of real-analytic Clifford-valued functions on the real line and the module of axially monogenic functions in $ \mathbb{R}^{m+1}$. This is provided by a particular case of the generalized CK-extension, see \cite{green}.

\begin{thm}
\label{GCK1}
Let $f_{0}(x_0)$ be a Clifford-valued analytic function in a real domain $ \Omega_1 \subset \mathbb{R}$. Then, there exists a unique sequence $ \{f_j( x_0)\}_{j=1}^\infty$ of analytic functions on $\Omega_1$ such that the series
$$ f(x_0,\underline{x})= \sum_{j=0}^\infty \underline{x}^j f_{j}(x_0),$$
is convergent in an axially symmetric slice $(m+1)$-dimensional neighbourhood $ \Omega \subset \mathbb{R}^{m+1}$ of $ \Omega_1$ and its sum is a monogenic function i.e., $(\partial_{x_0}+\partial_{\underline{x}})f(x_0, \underline{x})=0$.
\\Furthermore, the sum $f$ is formally given by the expression
\begin{equation}
\label{Exx}
f(x_0, \underline{x})= \Gamma \left(\frac{m}{2}\right) \left( \frac{|\underline{x}|\partial_{x_0}}{2} \right)^{- \frac{m}{2}} \left(\frac{|\underline{x}|\partial_{x_0}}{2} J_{\frac{m}{2}-1}\left( |\underline{x}|\partial_{x_0} \right)+ \frac{\underline{x} \partial_{x_0}}{2} J_{\frac{m}{2}}\left( |\underline{x}|\partial_{x_0} \right) \right) f_{0}(x_0),
\end{equation}
where $J_{\nu}$ is the Bessel function of the first kind of order $\nu$. Formula \eqref{Exx} is known as the generalized CK-extension of $f_0$, and it is denoted by $GCK[f_0](x_{0}, \underline{x})$.
\end{thm}

A connection between the modules of slice monogenic functions and axial monogenic functions is given by the well known Fueter-Sce-Qian's theorem. 
\\ This results was originally established by Fueter (see \cite{F}) in the quaternionic setting. Later M. Sce generalized it to the Euclidean space $ \mathbb{R}^{m+1}$ for odd values of the dimension $m$. Precisely, he proved that the differential operator 
$$\Delta^{\frac{m-1}{2}}_{\mathbb{R}^{m+1}}, \qquad (\hbox{where} \, m \in \mathbb{N} \, \hbox{is odd})$$
maps slice monogenic functions into axial monogenic functions, see \cite{S, CSS2}.
\\Later in 1997 T. Qian extended this result to any dimension $m \in \mathbb{N}$ by means of the Fourier multiplier.
\\ Now, we can summarize the Fueter-Sce's theorem as follows.
\begin{thm}[Fueter-Sce theorem]
Let $f(u+iv)= \alpha(u,v)+i \beta(u,v)$ be an intrinsic holomorphic function defined on an intrinsic complex domain $ \Omega_2 \subset \mathbb{C}$, and set $ \underline{\omega}= \frac{\underline{x}}{r}$ with $r=|\underline{x}|$. Then,
$$ \Delta^{\frac{m-1}{2}}_{\mathbb{R}^{m+1}}[f(x_0+ \underline{x})]= \Delta^{\frac{m-1}{2}}_{\mathbb{R}^{m+1}}[\alpha(x_0,r)+ \underline{\omega} \beta(x_0,r)],$$
is axially monogenic in the axially symmetric slice domain $\Omega=\Omega_2 \times \mathbb{S}^{m-1}=\{(x_0, \underline{x}) \in \mathbb{R}^{m+1}\,: \, (x_0, |\underline{x}|) \in \Omega_2\}.$
\end{thm} 
The Fueter-Sce's theorem can be rewritten by using the fact that every intrinsic holomorphic function is the unique holomorphic extension of a real-analytic function on the real line.

\begin{thm}[Fueter-Sce theorem]
\label{FSQ2}
Let $ \Omega_1 \subset \mathbb{R}$ be a real domain and $f_0 \in \mathcal{A}(\Omega_1) \otimes \mathbb{R}_m$. Then $\Delta^{\frac{m-1}{2}}_{\mathbb{R}^{m+1}} \circ  S[f_0](x_0, \underline{x})$ is an axial monogenic function on a $(m+1)$-dimensional axially symmetric slice neighbourhood $\Omega \subset \mathbb{R}^{m+1}$ of $ \Omega_1$.
\end{thm}

Hence by combining Theorem \ref{sliceex} and Theorem \ref{FSQ2} we have the following mapping property
$$ \Delta^{\frac{m-1}{2}}_{\mathbb{R}^{m+1}}: \mathcal{SM}(\Omega) \to \mathcal{AM}(\Omega).$$ 

We observe that generalized $CK$-extension does not coincide with the Fueter-Sce-Qian's map $\Delta^{\frac{m-1}{2}}_{\mathbb{R}^{m+1}} \circ S$ since the generalized $CK$-extension is an isomorphism between right modules while $\Delta^{\frac{m-1}{2}}_{\mathbb{R}^{m+1}} \circ S$ is not, see \cite{CSSOinverse, DKQS}. However, in \cite{DDG} has been proved a relation between these two extension tools.

\begin{thm}
\label{FG}
Let $ D \subset \mathbb{C}$ be an intrinsic complex domain. Consider a holomorphic function $f: D \to \mathbb{C}$ such that its restriction to the real line is real valued. Then for $m$ odd we have
\begin{equation}
\label{CGK5}
\Delta^{\frac{m-1}{2}}_{\mathbb{R}^{m+1}}f(x_0+ \underline{x})=  \gamma_m GCK[f^{(m-1)}(x_{0})],
\end{equation}
where $\gamma_m:= \frac{(-1)^{\frac{m-1}{2}}2^{m-1}}{(m-1)!} \left[\Gamma\left(\frac{m+1}{2}\right)\right]^2.$ Setting $\Omega_1:=D \cap \mathbb{R}$ we have the following commutative diagram
\begin{equation}\label{Diag1}
\begin{tikzcd}[row sep = 3em, column sep = 6em]
\mathcal{A}(\Omega_1) \otimes \mathbb{R}_n \arrow[r, "S", rightarrow] \arrow[d, "\gamma_m \partial_{x_0}^{m-1}", labels=left] & \mathcal{SM}(\Omega) \arrow[d, "\Delta^{\frac{m-1}{2}}_{\mathbb{R}^{m+1}}"] \\
\mathcal{A}(\Omega_1) \otimes \mathbb{R}_n \arrow[r, "\textup{GCK}", rightarrow]                  &\mathcal{AM}(\Omega)
\end{tikzcd}
\end{equation}
\end{thm}
\begin{rem}
In \cite{DDG} the authors showed the previous result for any dimension $m$.
\end{rem}

In this paper we are interested in the following class of functions
\begin{definition}[Axially harmonic function]
	\label{aah}
Let $\Omega \subset \mathbb{R}^{m+1}$ be an axially symmetric slice domain and let $\Omega_2 \subset \mathbb{C}$ be an intrinsic complex domain such that $\Omega=\Omega_2 \times \mathbb{S}^{m-1}$. A function $f \in \mathcal{A}(\Omega) \otimes \mathbb{R}_m$ is axially harmonic if it is of the form \eqref{form} and if it is in the kernel of the operator $ \Delta_{\mathbb{R}^{m+1}}$, i.e. $ \Delta_{\mathbb{R}^{m+1}}f=0$. This set of functions is denoted by $ \mathcal{AH}(\Omega)$.
\end{definition}

This type of functions are related to the Fueter-Sce's theorem. Indeed, from the splitting of the Laplace operator in terms of $ \mathcal{D}$ and $ \overline{\mathcal{D}}$  and Theorem \ref{FSQ2} we have the following corollary of Futer-Sce's theorem.

\begin{cor}
	\label{fact}
	Let $ \Omega_1 \subset \mathbb{R}$ be a real domain and $f_0 \in \mathcal{A}(\Omega_1) \otimes \mathbb{R}_m$. Then for $m \geq 3$ we have that $g(x):=\Delta^{\frac{m-3}{2}}_{\mathbb{R}^{m+1}} \mathcal{D} \circ  S[f_0](x_0, \underline{x})$ is an  axially harmonic function on a $(m+1)$-dimensional axially symmetric slice neighbourhood $\Omega \subset \mathbb{R}^{m+1}$ of $ \Omega_1$.
	Moreover, by applying the operator $\overline{\mathcal{D}}$ to the function $g(x)$, one obtains an axially monogenic function.
\end{cor}
\begin{proof}
We have to show that $g \in \mathcal{AH}(\Omega)$. It is clear that the function $g$ is of axial type. Now, we show that $g$ is in the kernel of the Laplace operator in $\mathbb{R}^{m+1}$. By Theorem \ref{FSQ2} and the fact that $\Delta_{\mathbb{R}^{m+1}}$ is a real-valued operator we get
$$\Delta_{\mathbb{R}^{m+1}} g(x)=\Delta^{\frac{m-1}{2}}_{\mathbb{R}^{m+1}} \mathcal{D}\circ  S[f_0](x_0, \underline{x})=0.$$
Now, we show that $\overline{\mathcal{D}}g(x)$ is axially monogenic. By Lemma \ref{Newt}, and noting that $g(x)$ is of axial type, we conclude that $\overline{\mathcal{D}}g(x)$ is also of axial type. Moreover, by the fact that the Laplace operator is factorized by the operators $ \mathcal{D}$ and $\overline{\mathcal{D}}$ we get
$$ \overline{\mathcal{D}} g(x)= \Delta^{\frac{m-1}{2}}_{\mathbb{R}^{m+1}} \circ  S[f_0](x_0, \underline{x}).$$
By Theorem \ref{FSQ2} and the above equality, we deduce that $ \overline{\mathcal{D}}g(x) $ is axially monogenic.

\end{proof}
\vspace{-5mm}
Thus, by merging the results of Theorem \ref{sliceex} and Theorem \ref{fact}, we derive the following mapping property:
$$ \Delta^{\frac{m-3}{2}}_{\mathbb{R}^{m+1}} \mathcal{D}: \mathcal{SM}(\Omega) \to \mathcal{AH}(\Omega).$$ 
\newline
\newline
The aim of this paper is to factorize diagram \eqref{Diag1}. In order to do this we need to complete the missing arrows from the following diagram

\begin{equation}\label{Diag2bis}
\begin{tikzcd}[row sep = 3em, column sep = 6em]
\mathcal{A}(\Omega_1)\otimes \mathbb{R}_m\arrow[r, "S", rightarrow] \arrow[d, dashed, "?", labels=left] & \mathcal{SM}(\Omega) \arrow[d, "\Delta_{\mathbb{R}^{m+1}}^{\frac{m-3}{2}} \mathcal{D}"] \\
? \arrow[r, dashed, "?" rightarrow]    \arrow[d, dashed, "?", labels=left] & \mathcal{AH}(\Omega) \arrow[d, "\overline{\mathcal{D}}"]\\
\mathcal{A}(\Omega_1)\otimes \mathbb{R}_m \arrow[r, "GCK", rightarrow]                 & \mathcal{AM}(\Omega)\\
\end{tikzcd}
\end{equation}
We start to address the problem from the next section.

\section{Generalized Cauchy-Kovalevskaya extension for axially harmonic functions}
This section aims to explain the central left dashed horizontal arrow in diagram \eqref{Diag2bis}, utilizing the generalized CK-extension method for axially harmonic functions. To establish this result, we examine the conditions under which a set of analytic Clifford-valued functions $\{A_j(x_0)\}_{j \in \mathbb{N}_0}$ defined on an open subset $\Omega_1 \subset \mathbb{R}$ allows the series $\sum_{j=0}^\infty \underline{x}^j A_{j}(x_0)$ to converge in a neighborhood $\Omega \subset \mathbb{R}^{m+1}$ of $\Omega_1$, in such a way that the resulting function is axially harmonic. We will show that the domain of convergence $ \Omega$ is an axially symmetric $(m+1)$-dimensional neighbourhood around $\Omega_1$. This situation is summarized in the following generalized Cauchy-Kovalevskaya property for axially harmonic functions.


\begin{thm}
\label{GCK}
Let $A_0$ and $A_1$ be two analytic Clifford-valued functions of one real variable $x_0$, defined in an open subset $ \Omega_1$ of the real line. Then there exist a unique sequence of functions $ \{A_j\}_{j \in \mathbb{N}_0}$ such that the series
\begin{equation}
\label{for}
f(x_0, \underline{x})= \sum_{j=0}^\infty \underline{x}^j A_j(x_0)
\end{equation}
converges in an axially symmetric $(m+1)$-dimensional neighborhood $ \Omega \subset \mathbb{R}^{m+1}$ of $ \Omega_1$ and such that $f(x_0, \underline{x})$ is harmonic (i.e. $\Delta_{\mathbb{R}^{m+1}} f(x_0, \underline{x})=0$).
\\ Moreover,
\begin{equation}
\label{harmo}
f(x_0, \underline{x})= \Gamma\left(\frac{m}{2}\right) \left(\frac{|\underline{x}| \partial_{x_0}}{2}\right)^{- \frac{m}{2}} \left[\left(\frac{|\underline{x}| \partial_{x_0}}{2}\right) J_{\frac{m}{2}-1}(|\underline{x}| \partial_{x_0})[A_0(x_0)]+ \frac{m \underline{x}}{2} J_{\frac{m}{2}}\left(| \underline{x}| \partial_{x_0}\right)[A_1(x_0)]\right],
\end{equation}
and the initial functions $A_0$ and $A_1$ can be recovered by
$$ f(x_0,\underline{x})\biggl|_{\underline{x}=0}=A_0(x_0),$$
\begin{equation}
	\label{ini}
 -\frac{1}{m} \partial_{\underline{x}} [f(x_0, \underline{x})]\biggl|_{\underline{x}=0}=A_1(x_0)
\end{equation}
The function in equation \eqref{harmo} is the harmonic generalized CK-extension of the couple $(A_0,A_1)$, and it is denoted by $HGCK[A_0,A_1]$.
\end{thm}
\begin{proof}
Given the two initial functions $A_0$ and $A_1$, we must prove the existence and uniqueness of the harmonic extension \eqref{harmo}.
\\ Let us start by showing the uniqueness, i.e. we will show that if \eqref{harmo} exists, then it is unique. Indeed, by applying the Laplace operator $ \Delta_{\mathbb{R}^{m+1}}:= \partial_{x_0}^2+ \Delta_{\underline{x}}$ to both side of \eqref{for}, and using formula \eqref{actla}, we get
\begin{eqnarray*}
\nonumber
\Delta_{\mathbb{R}^{m+1}} f(x_0, \underline{x}) &=& (\partial_{x_0}^2+ \Delta_{\underline{x}}) \left( \sum_{j=0}^\infty \underline{x}^j A_j(x_0)\right)\\
\nonumber
&=& \sum_{j=0}^\infty \underline{x}^j \partial_{x_0}^2 A_j(x_0)+ \sum_{j=2}^\infty \Delta_{\underline{x}}(\underline{x}^j) A_j(x_0)\\
\nonumber
&=& \sum_{j=0}^\infty \underline{x}^j \partial^2_{x_0} A_j(x_0)+\sum_{j=0}^\infty \Delta_{\underline{x}}[\underline{x}^{j+2}] A_{j+2}(x_0)\\
\label{step0}
&=& \sum_{j=0}^\infty \underline{x}^j [\partial_{x_0}^2 A_j(x_0)+c(m,j+2)A_{j+2}(x_0)],
\end{eqnarray*}
where the constant $c(m,j+2)$ is defined in \eqref{actla}.
From the condition $ \Delta_{\mathbb{R}^{m+1}} f(x_0, \underline{x})=0$ we get the recursion formula
\begin{equation}
\label{NN}
 A_{j+2}(x_0)= -\frac{1}{c(m, j+2)} \partial_{x_0}^2 A_{j}(x_0).
\end{equation}
Now, we use formula \eqref{NN} recursively for even indexes and we get  
$$ A_{2j}(x_0)=- \frac{\partial_{x_0}^2 A_{2(j-1)}(x_0)}{c(m, 2j)}=...= \frac{(-1)^j \partial_{x_0}^{2j}A_0(x_0)}{\prod_{k=1}^j c(m,2k)}.$$
Since $\prod_{k=1}^j c(m,2k)= \prod_{k=1}^j -4k \left(\frac{m}{2}+k-1\right)=(-4)^j j! \frac{\Gamma \left(\frac{m}{2}+j\right)}{\Gamma \left(\frac{m}{2}\right)}$ we get
\begin{equation}
\label{even0}
A_{2j}(x_0)= \frac{\Gamma\left(\frac{m}{2}\right)}{2^{2j} j! \Gamma \left(\frac{m}{2}+j\right)} \partial_{x_0}^{2j} A_{0}(x_0).
\end{equation}
Similarly if we use formula \eqref{NN} recursively for odd values we obtain
$$ A_{2j+1}(x_0)=- \frac{\partial_{x_0}^2 A_{2j-1}(x_0)}{c(m, 2j+1)}=...= \frac{(-1)^j \partial_{x_0}^{2j}A_1(x_0)}{\prod_{k=1}^{j} c(m,2k+1)}.$$
Since $\prod_{k=1}^{j} c(m,2k+1)= \prod_{k=1}^{j} -4k \left(\frac{m}{2}+k\right)=(-4)^j j! \frac{\Gamma\left(\frac{m}{2}+j+1\right)}{\Gamma\left(\frac{m}{2}+1\right)}$ we have 
\begin{equation}
\label{even01}
A_{2j+1}(x_0)= \frac{\Gamma\left(\frac{m}{2}+1\right)}{2^{2j} j! \Gamma \left(\frac{m}{2}+j+1\right)} \partial_{x_0}^{2j} A_{1}(x_0).
\end{equation}
Now, we substitute \eqref{even0} and \eqref{even01} on \eqref{for} and we get
\begin{eqnarray}
\nonumber
f(x_0, \underline{x})&=& \sum_{j=0}^{\infty} \underline{x}^j A_{2j}(x_0)+ \sum_{j=0}^{\infty} \underline{x}^{2j+1} A_{2j+1}(x_0)\\
\nonumber
&=&\sum_{j=0}^\infty \underline{x}^{2j} \frac{\Gamma\left(\frac{m}{2}\right)}{2^{2j} j! \Gamma \left(\frac{m}{2}+j\right)} \partial_{x_0}^{2j} A_{0}(x_0)+ \sum_{j=0}^\infty \underline{x}^{2j+1} \frac{\Gamma\left(\frac{m}{2}+1\right)}{2^{2j} j! \Gamma \left(\frac{m}{2}+j+1\right)} \partial_{x_0}^{2j} A_{1}(x_0)\\
\label{star1}
&=& \Gamma \left(\frac{m}{2}\right) \left[ \sum_{j=0}^\infty \frac{(-1)^j|\underline{x}|^{2j} \partial_{x_0}^{2j}}{2^{2j} j! \Gamma \left(\frac{m}{2}+j\right)}[A_0](x_0)+ \underline{x} \frac{m}{2}\sum_{j=0}^\infty \frac{(-1)^j|\underline{x}|^{2j} \partial_{x_0}^{2j}}{2^{2j} j! \Gamma \left(\frac{m}{2}+j+1\right)}[A_1](x_0)\right].
\end{eqnarray}
Using the defining expansion series of the Bessel functions $ \left(\frac{z}{2}\right)^{-\nu} J_{\nu}(z)= \sum_{j=0}^\infty \frac{(-1)^j z^{2j}}{2^{2j}j! \Gamma(\nu+j+1)}$, with   $\nu=\frac{m}{2}-1$ in the first sum, $\nu=\frac{m}{2}$ in the second sum and $z=| \underline{x}| \partial_{x_0}$ in both series, we get the formula in the statement.
\\ Let us show the existence of such a function in \eqref{harmo}, i.e., we have to show that the series in \eqref{star1} converges in a $(m+1)-$ dimensional neighbourhood of $\Omega_1$. To this end it is enough to show the convergence of the series
\begin{equation}
\label{series1}
\sum_{j=0}^\infty \frac{\underline{x}^{2j} \partial_{x_0}^{2j}[A_0](x_0) }{2^{2j} j! \Gamma \left( \frac{m}{2}+1\right)}.
\end{equation}
The convergence of the other series involved in \eqref{star1} follows from similar arguments. Now, we recall that $A_0(x_0)$ is a real analytic function. Thus, for any compact set $K \subset \Omega$, there exists constants $C_K$, $ \lambda_k >0$ such that
$$ | \partial_{x_0}^{2j}[A_0](x_0)| \leq C_K (2j)! \lambda_k^{2j}, \qquad \forall j \in \mathbb{N}_0, \quad x_0 \in K.$$ 
Hence, for $x_0 \in K$ we have
$$ \sum_{j=0}^\infty \biggl| \frac{\underline{x}^{2j} \partial_{x_0}^{2j}[A_0](x_0)}{2^{2j} j! \Gamma \left( \frac{m}{2}+1\right)} \biggl| \leq C_K \sum_{j=0}^\infty \frac{|\underline{x}|^{2j} (2j)! \lambda_k^{2j}}{2^{2j} j! \Gamma \left( \frac{m}{2}+j\right)}.$$
If we set $ \ell_j:=\frac{|\underline{x}|^{2j} (2j)! \lambda_k^{2j}}{2^{2j} j! \Gamma \left( \frac{m}{2}+j\right)}$, by applying the ratio test to the series in the right-hand side of the above inequality, we get
$$ \lim_{j \to \infty} \frac{\ell_{j+1}}{\ell_{j}}=\lim_{j \to \infty} \frac{| \underline{x}|^2 (2j+2)(2j+1) \lambda_k^2}{4 (j+1) \left(  \frac{m}{2}+j\right)}=|\underline{x}|^2 \lambda_k^2.$$
Thus, the series  in \eqref{series1} converges if $ | \underline{x}| < \frac{1}{\lambda_k}$. Therefore the series in \eqref{star1} converges in an axially symmetric $(m+1)$- dimensional neighbourhood $ \Omega$ of $ \Omega_1$.
\end{proof}

The above theorem defines a one-to-one correspondence between $( \mathcal{A}_1(\Omega_1) \otimes \mathbb{R}_m)^2$ and the space $ \mathcal{AH}(\Omega)$ of harmonic functions of axial type in $ \Omega \subset \mathbb{R}^{m+1}$. This mapping is denoted by
$$ HGCK: ( \mathcal{A}_1(\Omega_1) \otimes \mathbb{R}_m)^2 \to \mathcal{AH}(\Omega),$$
where $( \mathcal{A}_1(\Omega_1) \otimes \mathbb{R}_m)^2$ denotes two analytic functions. The mapping HGCK fits with the diagram \eqref{Diag2bis}. 

\begin{rem}
We can split the harmonic generalized CK-extension as a sum of two harmonic generalized CK-extensions, one gives as a result a scalar result while the other gives a 1-vector. Precisely, we have
$$ HGCK[A_0,A_1]=HGCK[A_0,0]+HGCK[0,A_1].$$ 
\end{rem}

We can write the CK-extension of axially harmonic functions in terms of the zero and one vector parts of the generalized CK-extension of axially monogeinc functions.

\begin{prop}
\label{split}
Let $A_0$, $A_1$ be two analytic functions of one real variable $x_0$, defined in on open subset $ \Omega_1$ of the real line. Then we can write the CK-extension of axially harmonic functions in terms of the generalized CK-extension of axially monogeinc functions, i.e.
\begin{equation}
\label{H1}
HGCK[A_0,A_1]=[GCK[A_0]]_0+ m GCK[[\mathcal{A}_1]]_1,
\end{equation}   
where $ \mathcal{A}_1$ is the primitive of $A_1$.
Moreover, we can write the generalized CK-exstension for axially monogenic functions in the following way
\begin{equation}
	\label{H2}
GCK[A_0]=HGCK[A_0,0]+ \frac{1}{m} HGCK[0, \partial_{x_0}A_0].
\end{equation}
\end{prop}
\begin{proof}
	We start proving \eqref{H1}. By \eqref{Exx} it is clear that
\begin{equation}
\label{H4}
[GCK[A_0]]_0=\Gamma \left(\frac{m}{2}\right) \left( \frac{|\underline{x}|\partial_{x_0}}{2} \right)^{- \frac{m}{2}} \frac{|\underline{x}|\partial_{x_0}}{2} J_{\frac{m}{2}-1}\left( |\underline{x}|\partial_{x_0} \right).
\end{equation}
By using another time \eqref{Exx}, the definition of the Bessel functions and the fact that $ \mathcal{A}_1(x_0)$ is a primitive of $A_1(x_0)$ we have
\begin{eqnarray}
	\nonumber
m GCK[[\mathcal{A}_1]]_1&=&m \Gamma \left(\frac{m}{2}\right) \left( \frac{|\underline{x}|\partial_{x_0}}{2} \right)^{- \frac{m}{2}} \left( \frac{\underline{x}}{2} J_{\frac{m}{2}} \left(| \underline{x}| \partial_{x_0} \right)\right)[\mathcal{A}_1](x_0)\\
\nonumber
&=&\Gamma\left(\frac{m}{2}\right) m \frac{\underline{x}}{2} \sum_{j=0}^\infty \frac{(-1)^j | \underline{x}|^{2j} \partial_{x_0}^{2j}[A_1](x_0)}{2^{2j }j! \Gamma \left(\frac{m}{2}+1+j\right)}\\
\label{H3}
&=& \Gamma\left(\frac{m}{2}\right) \left(\frac{|\underline{x}| \partial_{x_0}}{2}\right)^{- \frac{m}{2}} \left(\frac{m \underline{x}}{2} J_{\frac{m}{2}}\left(| \underline{x}| \partial_{x_0}\right)\right)[A_1(x_0)].
\end{eqnarray}
By making the sum of \eqref{H4} and \eqref{H3}, the result follows by \eqref{harmo}.
\\ Now, we prove \eqref{H2}. By \eqref{harmo} it is clear that
\begin{equation}
\label{H6}
	HGCK[A_0,0]=\Gamma \left(\frac{m}{2}\right) \left( \frac{|\underline{x}|\partial_{x_0}}{2} \right)^{- \frac{m}{2}} \frac{|\underline{x}|\partial_{x_0}}{2} J_{\frac{m}{2}-1}\left( |\underline{x}|\partial_{x_0} \right)[A_0(x_0)].
\end{equation}
By using another time \eqref{harmo} and the definition of the Bessel functions we have
\begin{eqnarray}
\nonumber
\frac{1}{m} HGCK[0, \partial_{x_0}A_0]&=&\Gamma \left(\frac{m}{2}\right) \left( \frac{|\underline{x}|\partial_{x_0}}{2} \right)^{- \frac{m}{2}} \left( \frac{\underline{x}}{2} J_{\frac{m}{2}} \left(| \underline{x}| \partial_{x_0} \right)\right)[\partial_{x_0} A_0](x_0)\\
\nonumber
&=& \Gamma \left(\frac{m}{2}\right) \frac{\underline{x}}{2} \partial_{x_0}\sum_{j=0}^\infty \frac{(-1)^j | \underline{x}|^{2j} \partial_{x_0}^{2j}[A_0](x_0)}{2^{2j }j! \Gamma \left(\frac{m}{2}+1+j\right)}\\
\label{H5}
&=& \Gamma\left(\frac{m}{2}\right) \left(\frac{|\underline{x}| \partial_{x_0}}{2}\right)^{- \frac{m}{2}} \left(\frac{ \underline{x}}{2} \partial_{x_0} J_{\frac{m}{2}}\left(| \underline{x}| \partial_{x_0}\right)\right)[A_0(x_0)].
\end{eqnarray}
Finally\eqref{H2} follows by making the sum of \eqref{H5} and \eqref{H6}, and by using \eqref{Exx}.
\end{proof}

Now, we present some examples of axially harmonic functions constructed using the generalized CK-extension for axially harmonic functions.

\begin{ex}
We consider as initial functions $A_0(x_0)=A_1(x_0)=e^{x_0}$. We want to use formula \eqref{H1}. Thus we observe that $ \mathcal{A}_1(x_0)= e^{x_0}$. By \cite[Remark 2.1]{DESS0} we know that
\begin{equation}
\label{exp}
GCK[A_1]= GCK[\mathcal{A}_1]= \Gamma \left( \frac{m}{2}\right) 2^{m/2-1}  \left( \widetilde{J}_{\frac{m}{2}-1}(|\underline{x}|)+\widetilde{J}_{\frac{m}{2}}(| \underline{x}|) \underline{x}\right) e^{x_0},
\end{equation}
where $\widetilde{J}_{\nu}(\rho)= \rho^{-\nu} J_{\nu}(\rho)$ being the Bessel function. Thus by \eqref{H1} we have
$$ HGCK[A_0,A_1]= \Gamma \left( \frac{m}{2}\right) 2^{m/2-1}  \left( \widetilde{J}_{\frac{m}{2}-1}(|\underline{x}|)+m\widetilde{J}_{\frac{m}{2}}(| \underline{x}|) \underline{x}\right) e^{x_0}.$$
\end{ex}

\begin{ex}
We consider as initial functions $A_0(x_0)=e^{x_0}$  and $A_1(x_0)=-2x_0 e^{-x_0^2}$. A primitive of $A_1(x_0)$ is given by $ \mathcal{A}_1(x_0)= e^{-x_0^2}$. By \cite{DESS0} we know that
\begin{eqnarray}
\label{herm}
GCK[A_1]&=& \Gamma \left(\frac{m}{2}\right) \left( \sum_{\ell=0}^\infty \frac{(-1)^\ell | \underline{x}|^{2 \ell} H_{2 \ell}(x_0)}{\Gamma \left(\frac{m}{2}+ \ell\right) 2^{2 \ell } \ell !} \right) e^{- \frac{x_0^2}{2}}\\
\nonumber
&&-\Gamma \left(\frac{m}{2}\right) \left( \sum_{\ell=0}^\infty \frac{(-1)^\ell | \underline{x}|^{2 \ell+1} H_{2 \ell+1}(x_0)}{\Gamma \left(\frac{m}{2}+ \ell+1\right) 2^{2 \ell+1 } \ell !} \right)  \underline{\omega}e^{- \frac{x_0^2}{2}}, \qquad \underline{\omega}=\frac{\underline{x}}{|\underline{x}|},
\end{eqnarray}
where $H_n(x_0)$, $n \in \mathbb{N}$, are the Hermite polynomials given by $H_n(x_0)=n! \sum_{i=0}^{ \lfloor \frac{n}{2} \rfloor} \frac{(-1)^i x_0^{n-2i}}{i! 2^i (n-2i)!}$.  Thus by \eqref{H1}, \eqref{exp} and \eqref{herm} we have 	
$$ HCGK[A_0,A_1]=\Gamma \left(\frac{m}{2}\right)\left[ 2^{m/2-1}  \widetilde{J}_{\frac{m}{2}-1}(|\underline{x}|)e^{x_0}-m\sum_{\ell=0}^\infty \frac{(-1)^\ell | \underline{x}|^{2 \ell+1} H_{2 \ell+1}(x_0)}{\Gamma \left(\frac{m}{2}+ \ell+1\right) 2^{2 \ell+1 } \ell !}   \underline{\omega}e^{- \frac{x_0^2}{2}} \right].$$
\end{ex}

\begin{ex}
We consider as initial functions $A_0(x_0)=e^{-x_0^2}$  and $A_1(x_0)=e^{x_0}$.	It is clear that a primitive of $A_1(x_0)$ is given by $ \mathcal{A}_1(x_0)=e^{x_0}$. Thus by \eqref{H1}, \eqref{exp} and \eqref{herm} we have 	
$$ HCGK[A_0,A_1]=\Gamma \left(\frac{m}{2}\right)\left[\sum_{\ell=0}^\infty \frac{(-1)^\ell | \underline{x}|^{2 \ell} H_{2 \ell}(x_0)}{\Gamma \left(\frac{m}{2}+ \ell\right) 2^{2 \ell } \ell !} e^{- \frac{x_0^2}{2}}+m\widetilde{J}_{\frac{m}{2}}(| \underline{x}|) \underline{x} e^{x_0} \right].$$
\end{ex}
The results of this section together with Theorem \ref{FG} will be useful in the Section 5 to establish a connection between the generalized CK-extension for axially harmonic functions and the Fueter-Sce's theorem.

\section{Plane waves decomposition}

The generalized CK-extension for axially monogenic functions can be expressed in terms of integrals over the sphere $\mathbb{S}^{m-1}$ of functions of plane wave type, see \cite{DESS}. These are functions depending on the inner product $ \langle \underline{x}, \underline{\omega} \rangle$, where the 1-vector variable $ \underline{\omega}$ is independent of $ \underline{x}$. 

\begin{thm}
	Let $ \Omega_1 \subset \mathbb{R}$ be a real domain and let $f_0 \in \mathcal{A}(\Omega_1) \otimes \mathbb{R}_m$. Then
	$$ GCK[f_0](x_0, \underline{x})= \frac{\Gamma\left(\frac{m}{2}\right)}{2 \pi^{\frac{m}{2}}} \left( \int_{\mathbb{S}^{m-1}} exp(\langle \underline{\omega}, \underline{x} \rangle \underline{\omega} \partial_{x_0}) dS_{\underline{\omega}}\right) f_0(x_0).$$
\end{thm}

The aim of this section is to get a plane waves decomposition of the harmonic generalized CK-extension. We shall make use of the Funk-Hecke theorem. In particular, we need the following particular case, see \cite{Ho}. 
\begin{thm}[Funk-Hecke]
	\label{FH}
	Let $k \in \mathbb{N}_0$. We assume that $P_k$ is a homogeneous polynomial of degree $k$ then we have
	$$ \int_{\mathbb{S}^{m-1}} \langle \underline{x}, \underline{\omega} \rangle^j P_k(\underline{\omega}) dS_{\underline{\omega}}= \frac{j!}{(j- k)!} \frac{2 \pi^{\frac{m-1}{2}}}{2^k} \frac{\Gamma \left(\frac{j-k+1}{2}\right)}{\Gamma \left( \frac{m+j+k}{2}\right)} | \underline{x}|^{j-k} P_k(\underline{x}).$$
\end{thm}
\begin{thm}
\label{plane}
	Let $ f(x_0, \underline{x})= \sum_{j=0}^\infty \underline{x}^j A_j(x_0)$ satisfying the same conditions of Theorem \ref{GCK}. Then $f(x_0, \underline{x})$ can be decomposed into plane waves as follows
	$$ f(x_0, \underline{x})= \frac{\Gamma\left(\frac{m}{2}\right)}{2 \pi^{\frac{m}{2}}} \int_{\mathbb{S}^{m-1}}\left[\cosh\left(\langle \underline{x}, \underline{\omega} \rangle \underline{\omega} \partial_{x_0}\right)[A_0](x_0)+m \sinh \left(\langle \underline{x}, \underline{\omega} \rangle \underline{\omega} \partial_{x_0}\right)[A_1](x_0)\right] dS_{\underline{\omega}}.$$ 
\end{thm}
\begin{proof}
	Applying the Funk-Hecke theorem (see Theorem \ref{FH}) with $k=0$ and $k=1$ respectively, we obtain the following equalities
	\begin{equation}
		\label{P1}
		\underline{x}^{2j}=(-1)^j | \underline{x}|^{2j}= \frac{(-1)^j \Gamma\left(\frac{m}{2}+j\right)}{2 \pi^{\frac{m-1}{2}} \Gamma\left( j+ \frac{1}{2}\right)} \int_{\mathbb{S}^{m-1}} \langle \underline{x}, \underline{\omega} \rangle^{2j} dS_{\underline{\omega}},
	\end{equation}
	\begin{equation}
		\label{P2}
		\underline{x}^{2j+1}=(-1)^j \underline{x}| \underline{x}|^{2j}= \frac{(-1)^j \Gamma\left(\frac{m}{2}+j+1\right)}{(2j+1)\pi^{\frac{m-1}{2}} \Gamma\left( j+ \frac{1}{2}\right)} \int_{\mathbb{S}^{m-1}} \langle \underline{x}, \underline{\omega} \rangle^{2j+1} \underline{\omega} dS_{\underline{\omega}}.
	\end{equation}
	Now by Theorem \ref{GCK}, \eqref{P1}, \eqref{P2} and using the expression of the Bessel functions we get
	\begin{eqnarray*}
		f(x_0, \underline{x})
		&=& \Gamma \left(\frac{m}{2}\right) \left[ \sum_{j=0}^\infty \frac{\underline{x}^{2j} \partial_{x_0}^{2j}}{2^{2j} j! \Gamma \left(\frac{m}{2}+j\right)}[A_0](x_0)+ \underline{x} \frac{m}{2}\sum_{j=0}^\infty \frac{\underline{x}^{2j} \partial_{x_0}^{2j}}{2^{2j} j! \Gamma \left(\frac{m}{2}+j+1\right)}[A_1](x_0)\right]\\
		&=& \Gamma \left(\frac{m}{2}\right) \biggl[\sum_{j=0}^\infty \frac{(-1)^j \partial_{x_0}^{2j} [A_0](x_0)}{2^{2j+1} j! \pi^{\frac{m-1}{2}} \Gamma \left( j+ \frac{1}{2}\right)} \int_{\mathbb{S}^{m-1}} \langle \underline{x}, \underline{\omega} \rangle^{2j} dS_{\underline{\omega}}+ \\
		&&+ \frac{m}{2} \sum_{j=0}^\infty \frac{(-1)^j \partial_{x_0}^{2j}[A_1](x_0)}{2^{2j} j!(2j+1) \pi^{\frac{m-1}{2}} \Gamma \left(j+ \frac{1}{2}\right)} \int_{\mathbb{S}^{m-1}} \langle \underline{x}, \underline{\omega} \rangle^{2j+1} \underline{\omega} dS_{\underline{\omega}} \biggl].
	\end{eqnarray*}	
	Using the identities $ 2 \Gamma\left( j+ \frac{3}{2}\right)= (2j+1) \Gamma \left(j+ \frac{1}{2}\right)$, $ 2^{2j} j! \Gamma \left(j+ \frac{1}{2}\right)= \pi^{\frac{1}{2}} (2j)!$ and $ 2^{2j+1} j! \Gamma \left(j+ \frac{3}{2}\right)= \pi^{\frac{1}{2}} (2j+1)!$ we get
	\begin{eqnarray*}
		f(x_0, \underline{x})&=& c_m \int_{\mathbb{S}^{m-1}} \left( \sum_{j=0}^\infty \frac{(-1)^j \langle \underline{x}, \underline{\omega} \rangle^{2j} \partial_{x_0}^{2j} [A_0](x_0)}{2^{2j} j! \Gamma \left(j+ \frac{1}{2}\right)}+ m\sum_{j=0}^\infty \frac{(-1)^j \langle \underline{x}, \underline{\omega} \rangle^{2j+1} \underline{\omega} \partial_{x_0}^{2j} [A_1](x_0)}{2^{2j+1} j! \Gamma \left(j+ \frac{3}{2}\right)} \right)dS_{\underline{\omega}}\\
		&=& c_m\int_{\mathbb{S}^{m-1}} \left( \sum_{j=0}^\infty \frac{(-1)^j \langle \underline{x}, \underline{\omega} \rangle^{2j} \partial_{x_0}^{2j} [A_0](x_0)}{(2j)!}+ m\sum_{j=0}^\infty \frac{(-1)^j \langle \underline{x}, \underline{\omega} \rangle^{2j+1} \underline{\omega} \partial_{x_0}^{2j} [A_1](x_0)}{(2j+1)!} \right) dS_{\underline{\omega}},
	\end{eqnarray*}
	where $c_m:=\frac{\Gamma \left(\frac{m}{2}\right)}{2 \pi^{\frac{m}{2}}} $.
	By using the fact that $ \underline{\omega}^{2j}=(-1)^j$, the above  power series of the differential operator $ \partial_{x_0}^2$ can be written as
	\begin{equation}
		\label{P3}
		\cosh (\langle  \underline{x}, \underline{\omega}\rangle \underline{\omega} \partial_{x_0} )[A_0](x_0)= \sum_{j=0}^\infty \frac{(-1)^j \langle \underline{x}, \underline{\omega} \rangle^{2j} \partial_{x_0}^{2j}[A_0](x_0)}{(2j)!},
	\end{equation}
	and
	\begin{equation}
		\label{P4}
		\sinh(\langle  \underline{x}, \underline{\omega}\rangle \underline{\omega} \partial_{x_0})[A_1](x_0)=\sum_{j=0}^\infty \frac{(-1)^j \langle \underline{x}, \underline{\omega} \rangle^{2j+1} \underline{\omega} \partial_{x_0}^{2j+1} [A_1](x_0)}{(2j+1)!},
	\end{equation}
	as power series on the variables $ \omega_1$,..., $ \omega_m$, where $x_0$ and $ \underline{x}$ are considered as parameters.
	Thus by \eqref{P3} and \eqref{P4} we get
	$$
	f(x_0, \underline{x}) 
	=\frac{\Gamma\left(\frac{m}{2}\right)}{2 \pi^{\frac{m}{2}}} \int_{\mathbb{S}^{m-1}}\left[\cosh\left(\langle \underline{x}, \underline{\omega} \rangle \underline{\omega} \partial_{x_0}[A_0](x_0)\right)+m \sinh \left(\langle \underline{x}, \underline{\omega} \rangle \underline{\omega} \partial_{x_0}\right)[A_1](x_0)\right] dS_{\underline{\omega}}.
	$$

\end{proof}

\section{Connection between the Fueter-Sce's theorem and the generalized CK-extension for axially harmonic functions }	

The purpose of this section is to complete diagram \eqref{Diag2bis}. In the previous section, we derived the central row of \eqref{Diag2bis}. To understand the remaining arrows, we need to explore the relationship between the Fueter-Sce theorem and the generalized CK-extension for axially harmonic functions. To proceed, we first review the action of the Laplace operator in $m+1$ variables on a real-valued harmonic function (see \cite{D}).

\begin{lemma}
\label{R1}
Let $h(x_0,r)$ be an $ \mathbb{R}$-valued harmonic function, i.e. $(\partial_{x_0}^2+ \partial_r^2)h(x_0,r)=0$ harmonic function. If we extend $h$ to $\mathbb{R}^{m+1}$ by taking $r=| \underline{x}|$ we obtain
$$ \Delta_{\mathbb{R}^{m+1}}^j h(x_0,r)= \prod_{\ell=1}^j(m-2 \ell+1) \left( \frac{1}{r} \partial_r\right)^j h(x_0,r)$$
and
$$ \Delta_{\mathbb{R}^{m+1}}^j \underline{\omega} h(x_0,r)= \prod_{\ell=1}^j(m-2 \ell+1) \left(  \partial_r\frac{1}{r}\right)^j \underline{\omega}h(x_0,r) ,$$
\end{lemma}

Another important tool that we will need in the sequel will be the following generalization of the classic Cauchy-Riemann system, see \cite{S, CSS2}. 

\begin{lemma}
\label{R2}
Let $f(z)= \alpha(u,v)+i \beta(u,v)$, with $z=u+iv$, be an intrinsic holomorphic function and consider
$$ \alpha_j:=\left( \frac{1}{r} \partial_r\right)^j \alpha(x_0,r) \quad  \hbox{and} \quad \beta_j:=\left(  \partial_r\frac{1}{r}\right)^j \beta(x_0,r),$$
then
$$ \begin{cases}
\partial_{x_0} \alpha_j= \partial_r \beta_j+ \frac{2 j}{r} \beta_j\\
\partial_r \alpha_j= - \partial_{x_0} \beta_j.
\end{cases}
$$ 
\end{lemma}
\begin{rem}
If we consider $j=0$ in Lemma \ref{R2} we get back the classical Cauchy-Riemann system, see \eqref{CR1}.
\end{rem}

The following result will be extremely useful to derive a connection between the Fueter-Sce's theorem and the generalized CK-extension for axially harmonic functions. 

\begin{lemma}
	\label{lappn}
Let $f(u+iv)= \alpha(u,v)+i \beta(u,v)$ be an intrinsic holomorphic function. Then for $ \ell \geq 0$ and $m$ being odd we have
\begin{equation}
\label{lapla1}
\Delta_{\mathbb{R}^{m+1}}^{\frac{m-2 \ell -1}{2}} f(x_0+ \underline{x})= 2^{\frac{m-2 \ell -1}{2}} \frac{\left( \frac{m-1}{2}\right)!}{\ell !} \left[ \left( \frac{1}{r} \partial_r\right)^{\frac{m-2 \ell -1}{2}} \alpha(x_0,r)+  \underline{\omega}\left(  \partial_r \frac{1}{r}\right)^{\frac{m-2 \ell -1}{2}} \beta(x_0,r)\right]
\end{equation} 
and
\begin{equation}
\label{lapla2}
\Delta_{\mathbb{R}^{m+1}}^{\frac{m-2 \ell -1}{2}} \mathcal{D} f(x_0+ \underline{x})=-2^{\frac{m-2 \ell +1}{2}} \frac{\left( \frac{m-1}{2}\right)!}{(\ell-1)!} \frac{1}{r} \left( \partial_r \frac{1}{r}\right)^{\frac{m-2 \ell -1}{2}} \beta(x_0, r).
\end{equation}
\end{lemma}
\begin{proof}
	Since by hypothesis the functions $f$ is holomorphic its components $\alpha$ and $\beta$ are harmonic. Thus by Lemma \ref{R1} we have that
\begin{equation}
\label{h1}
\Delta_{\mathbb{R}^{m+1}}^{\frac{m-2 \ell -1}{2}} \alpha(x_0,r)= \prod_{j=1}^{\frac{m-2 \ell -1}{2}}(m-2 j+1) \left( \frac{1}{r} \partial_r\right)^{\frac{m-2 \ell -1}{2}} \alpha(x_0,r)
\end{equation}
and
\begin{equation}
\label{h2}
 \Delta_{\mathbb{R}^{m+1}}^{\frac{m-2 \ell -1}{2}} \underline{\omega} \beta(x_0,r)= \prod_{j=1}^{\frac{m-2 \ell -1}{2}} (m-2 j+1) \left(  \partial_r\frac{1}{r}\right)^{\frac{m-2 \ell -1}{2}} \underline{\omega}\beta(x_0,r).
\end{equation}
Then by using the identity $ \prod_{j=1}^{\frac{m-2 \ell -1}{2}} (m-2 j+1)=2^{\frac{m-2 \ell -1}{2}} \frac{ \left( \frac{m-1}{2}\right)!}{\ell!}$, and making the sum of \eqref{h1} and \eqref{h2} we get formula \eqref{lapla1}. 
Now, we prove formula \eqref{lapla2}. By formulas \eqref{lapla1} and \eqref{Dirac} we obtain
\begin{eqnarray*}
 \mathcal{D} \Delta_{\mathbb{R}^{m+1}}^{\frac{m-2 \ell -1}{2}}f(x_0+\underline{x})&=& 2^{\frac{m-2 \ell -1}{2}} \frac{ \left( \frac{m-1}{2}\right)!}{\ell!} \mathcal{D} \left[ \left( \frac{1}{r} \partial_r\right)^{\frac{m-2 \ell -1}{2}} \alpha(x_0,r)+  \underline{\omega}\left(  \partial_r \frac{1}{r}\right)^{\frac{m-2 \ell -1}{2}} \beta(x_0,r)\right] \\
 &=& 2^{\frac{m-2 \ell -1}{2}} \frac{ \left( \frac{m-1}{2}\right)!}{\ell!} \left[\biggl( \partial_{x_0}\left( \frac{1}{r} \partial_r\right)^{\frac{m-2 \ell -1}{2}} \alpha(x_0,r)- \partial_r \left(  \partial_r \frac{1}{r}\right)^{\frac{m-2 \ell -1}{2}} \beta(x_0,r)+ \right.\\
\nonumber
&& \left.   - \frac{m-1}{r}   \left(\partial_r \frac{1}{r}\right)^{\frac{m-2 \ell -1}{2}} \beta(x_0,r)\biggl) \right.\\ 
&&\left. + \underline{\omega} \biggl( \partial_{x_0} \left(  \partial_r \frac{1}{r}\right)^{\frac{m-2 \ell -1}{2}} \beta(x_0,r)+ \partial_r \left( \frac{1}{r} \partial_r\right)^{\frac{m-2 \ell -1}{2}} \alpha(x_0,r) \biggl)\right].
\end{eqnarray*}
By Lemma \ref{R2} we have that
$$\partial_{x_0} \left(\partial_r \frac{1}{r}\right)^{\frac{m-2\ell-1}{2}}\alpha(x_0,r)=\partial_{r} \left(\partial_r \frac{1}{r}\right)^{\frac{m-2\ell-1}{2}}\beta(x_0,r)+ \frac{m-2\ell-1}{r} \left(\partial_r \frac{1}{r}\right)^{\frac{m-2\ell-1}{2}}.$$
$$ \partial_{x_0} \left( \partial_r \frac{1}{r} \right)^{\frac{m-2\ell-1}{2}} \beta(x_0,r)=-\partial_r  \left( \partial_r \frac{1}{r} \right)^{\frac{m-2\ell-1}{2}} \alpha(x_0,r).$$
Thus we have
\begin{eqnarray*}
 \mathcal{D} \Delta_{\mathbb{R}^{m+1}}^{\frac{m-2 \ell -1}{2}}f(x_0+\underline{x})&=&2^{\frac{m-2 \ell -1}{2}} \frac{ \left( \frac{m-1}{2}\right)!}{\ell!} \left[  \frac{m-2 \ell -1}{r} \left(\partial_r \frac{1}{r}\right)^{\frac{m-2\ell-1}{2}}\beta(x_0,r) \right.\\
 &&\left.- \frac{m-1}{r} \left(\partial_r \frac{1}{r}\right)^{\frac{m-2\ell-1}{2}} \beta(x_0,r) \right]\\
\nonumber
&=&  -2^{\frac{m-2 \ell +1}{2}} \frac{\left( \frac{m-1}{2}\right)!}{(\ell-1)!} \frac{1}{r} \left(\partial_r \frac{1}{r}\right)^{\frac{m-2\ell-1}{2}}\beta(x_0,r).
\end{eqnarray*}
\end{proof}
\begin{rem}
We observe that the power of the Laplace operator in \eqref{lapla1} and \eqref{lapla2} is an ineger positive number.
\end{rem}
Now, we have all the tools to formulate the main result of this section.
\begin{thm}
\label{Sceconn}
Let $f(u+iv)= \alpha(u,v)+i \beta(u,v)$ be an intrinsic holomorphic function defined on an intrinsic complex domain $ \Omega_2 \subset \mathbb{C}$. Then for $m \geq 3$ and odd we have
\begin{equation}
\label{App}
\Delta_{\mathbb{R}^{m+1}}^{\frac{m-3}{2}} \mathcal{D} [f(x_0+ \underline{x})]= \gamma_{m} HGCK[(f^{(m-2)}(x_0),0)],
\end{equation}
where $f^{(m-2)}$ denotes the $m-2$-th derivative.
Setting $\gamma_m$ like in Theorem \ref{FG} and $ \Omega_1= \Omega_2 \cap \mathbb{R}$, and $\Omega$ being an axially symmetric slice domain we obtain the following diagram
\begin{equation}\label{Diag2}
	\begin{tikzcd}[row sep = 3em, column sep = 6em]
		\mathcal{A}(\Omega_1)\otimes \mathbb{R}_m\arrow[r, "S", rightarrow] \arrow[d, "\gamma_m (\partial_{x_0}^{m-2} {,} 0)", labels=left] & \mathcal{SM}(\Omega) \arrow[d, "\Delta_{\mathbb{R}^{m+1}}^{\frac{m-3}{2}}\mathcal{D}"] \\
	(\mathcal{A}(\Omega_1)\otimes \mathbb{R}_m)^2 \arrow[r, "HGCK", rightarrow]    \arrow[d, "\partial_{x_0}P_1", labels=left] & \mathcal{AH}(\Omega) \arrow[d, "\overline{\mathcal{D}}"]\\
	\mathcal{A}(\Omega_1)\otimes \mathbb{R}_m \arrow[r, "GCK", rightarrow]                 & \mathcal{AM}(\Omega)\\
	\end{tikzcd}
\end{equation}
where $S$ is the slice operator, see \eqref{26} and $P_1$ is the projection of the first component.
\end{thm}
\begin{proof}
Since $f$ is an intrinsic holomorphic function we have that
\begin{eqnarray*}
f(u+iv)&=& \sum_{j=0}^\infty \frac{(iv)^j}{j!} f^{(j)}(u)\\
&=& \sum_{j=0}^\infty \frac{(-1)^jv^{2j}}{(2j)!} f^{(2j)}(u)+i \sum_{j=0}^\infty \frac{(-1)^j v^{2j+1}}{(2j+1)!} f^{(2j+1)}(u)\\
&=& \alpha(u,v)+i \beta(u,v).
\end{eqnarray*}
We focus on $ \beta(u,v)$ by setting $u=x_0$ and $v=r$. One can easily show by induction that for $s \geq 0$ we have
\begin{equation}
\label{1N}
\left( \partial_r \frac{1}{r}\right)^{s} r^{2j+1}=\frac{2^s j! }{(j-s)!} r^{2j-2s+1}, \qquad \left( \frac{1}{r} \partial_r\right)^{s} r^{2j}=\frac{2^s j! }{(j-s)!} r^{2j-2s} \quad j \geq s.
\end{equation}
Thus by \eqref{1N} we have
\begin{eqnarray}
\nonumber
\left( \partial_r \frac{1}{r}\right)^{s} \beta(x_0,r)
&=& \sum_{j= s}^\infty \frac{(-1)^j}{(2j+1)!} 2^{s} \frac{j!}{(j- s)!} r^{2j-2s+1} f^{(2j+1)}(x_0)\\
\label{star2}
&=& (-2)^s \sum_{j=0}^\infty \frac{(-1)^j (j+ s)!}{(2j+2 s+1)! j!} r^{2j+1} f^{(2j+2 s+1)}(x_0).
\end{eqnarray}
Now, by formula \eqref{lapla2}, with $ s=1$, and formula \eqref{star2} we get
\begin{eqnarray*}
\Delta_{\mathbb{R}^{m+1}}^{\frac{m-3}{2}} \mathcal{D}f(x_0+ \underline{x})&=&-2^{\frac{m-1}{2}} \left( \frac{m-1}{2}\right)! \frac{1}{r} \left( \partial_r \frac{1}{r}\right)^{\frac{m-3}{2}} \beta(x_0,r)\\
&=& -2^{\frac{m-1}{2}} \left( \frac{m-1}{2}\right)! (-2)^{\frac{m-3}{2}} \sum_{j=0}^\infty \frac{(-1)^j \left(j+ \frac{m-3}{2}\right)!}{(2j+m-2)! j!} r^{2j} f^{(2j+m-2)}(x_0)\\
&=&(-1)^{\frac{m-1}{2}} 2^{m-2} \left( \frac{m-1}{2}\right)!  \sum_{j=0}^\infty \frac{ \left(j+ \frac{m-3}{2}\right)!}{(2j+m-2)! j!} \underline{x}^{2j} f^{(2j+m-2)}(x_0).
\end{eqnarray*}
By the generalized harmonic CK-extension (see Theorem \ref{GCK}), The axially harmonic function described above is entirely determined by its restriction to the real line, along with the restriction of the result of applying $ \partial_{\underline{x}}$  to the function on the real line. In our case we have
\begin{eqnarray*}
\Delta_{\mathbb{R}^{m+1}}^{\frac{m-3}{2}} \mathcal{D} f(x_{0}+ \underline{x})|_{\underline{x}=0}&=&(-1)^{\frac{m-1}{2}} 2^{m-2} \left( \frac{m-1}{2}\right)! \frac{\left(\frac{m-3}{2}\right)!}{(m-2)!} f^{(m-2)}(x_0)\\
&=& \gamma_m f^{(m-2)}(x_0).
\end{eqnarray*}
By the fact that $ \partial_{\underline{x}} \underline{x}^{2j}=-2j \underline{x}^{2j-1}$, with $j \geq 1$, we get
$$ \partial_{\underline{x}}[\Delta_{\mathbb{R}^{m+1}}^{\frac{m-3}{2}} \mathcal{D} f(x_{0}+ \underline{x})]=(-1)^{\frac{m+1}{2}} 2^{m-1} \left( \frac{m-1}{2}\right)!  \sum_{j=1}^\infty \frac{ \left(j+ \frac{m-3}{2}\right)!}{(2j+m-2)! (j-1)!} \underline{x}^{2j-1} f^{(2j+m-2)}(x_0).$$
Thus we have
$$ -\frac{1}{m} \partial_{\underline{x}}[\Delta_{\mathbb{R}^{m+1}}^{\frac{m-3}{2}} \mathcal{D} f(x_{0}+ \underline{x})]|_{\underline{x}=0}=0.$$
This proves formula \eqref{App}.
\end{proof}
The previous theorem sheds light on the nature of the Fueter operator, which can be expressed in terms of Bessel functions of the first kind.
\begin{cor}
Let $f(u+iv)= \alpha(u,v)+i \beta(u,v)$ an intrinsic holomorphic function. Then
\begin{equation}
\label{FF}
\mathcal{D}[f(x_0+ \underline{x})]= \frac{\sqrt{\pi}}{2} \left( \frac{| \underline{x}| \partial_{x_0}}{2}\right)^{-1/2} J_{1/2}(|\underline{x}|\partial_{x_0}) \partial_{x_0}f(x_0).
\end{equation}
\end{cor}
\begin{proof}
We consider $m=3$ in Theorem \ref{Sceconn}, and we get
$$ \mathcal{D}[f(x_0+ \underline{x})]=-2 HGCK[f^{(1)}(x_0), 0].$$
Thus the expression \eqref{FF} follows by \eqref{harmo} and the fact that $\Gamma(\frac{1}{2})= \sqrt{\pi}$.
\end{proof}
In general axially harmonic functions can be also obtained in other ways insied the Fueter-Sce construction.

\begin{prop}
Let $ \Omega_1 \subset \mathbb{R}$ be a real domain and $f_0 \in \mathcal{A}(\Omega_1) \otimes \mathbb{R}_m$. Then for $m \geq 1$ we have that $[\Delta^{\frac{m-1}{2}}_{\mathbb{R}^{m+1}}  \circ  S[f_0](x_0, \underline{x})]_0$  and $[\Delta^{\frac{m-1}{2}}_{\mathbb{R}^{m+1}}  \circ  S[f_0](x_0, \underline{x})]_1$, are axially harmonic functions on a $(m+1)$-dimensional axially symmetric slice neighbourhood $\Omega \subset \mathbb{R}^{m+1}$ of $ \Omega_1$. We denote by $[.]_0$ and $[.]_1$ the scalar and one-vector parts of the operator $\Delta^{\frac{m-1}{2}}_{\mathbb{R}^{m+1}} $. Thus we have
$$
\begin{CD}
	&& \textcolor{black}{\mathcal{A}(\Omega_1)\otimes \mathbb{R}_m}  @>\ \     S >>\textcolor{black}{\mathcal{SM}(\Omega)}@>\ \     [\Delta^{\frac{m-1}{2}}_{m+1}]_k>>\textcolor{black}{\mathcal{AH}(\Omega_1)}, \qquad k \in \{0,1\}
\end{CD}
$$
where $\Omega$ is an axially symmetric set.
\end{prop}
\begin{proof}
	By Lemma \ref{lappn} it is clear that  $[\Delta^{\frac{m-1}{2}}_{\mathbb{R}^{m+1}}  \circ  S[f_0](x_0, \underline{x})]_0$  and $[\Delta^{\frac{m-1}{2}}_{\mathbb{R}^{m+1}}  \circ  S[f_0](x_0, \underline{x})]_1$ are of axial type.
By Theorem \ref{FSQ2} we know that $\Delta^{\frac{m-1}{2}}_{\mathbb{R}^{m+1}}  \circ  S[f_0](x_0, \underline{x})$ is of axial type i.e:
$$ \Delta^{\frac{m-1}{2}}_{\mathbb{R}^{m+1}}  \circ  S[f_0](x_0, \underline{x})=A(x_0,r)+ \underline{\omega}B(x_0,r).$$
Thus, we have show that $A(x_0,r)=[\Delta^{\frac{m-1}{2}}_{\mathbb{R}^{m+1}}  \circ  S[f_0](x_0, \underline{x})]_0$ and $\underline{\omega}B(x_0,r)=[\Delta^{\frac{m-1}{2}}_{\mathbb{R}^{m+1}}  \circ  S[f_0](x_0, \underline{x})]_1$ are harmonic. By \eqref{lapp} this means to prove that
\begin{equation}
\label{h11}
\partial_{x_0}^2 A(x_0,r)+ \partial_r^2 A(x_0,r)+ \frac{m-1}{r} \partial_r A(x_0,r)=0,
\end{equation}
\begin{equation}
\label{h12}
\partial_{x_0}^2 \underline{\omega}B(x_0,r)+ \partial_r^2 \underline{\omega}B(x_0,r)+ (m-1) \partial_r \left(\frac{\underline{\omega}B(x_0,r)}{r} \right)=0.
\end{equation}
We prove \eqref{h11}. By \eqref{vek} we have
\begin{eqnarray*}
\partial_{x_0}^2 A(x_0,r)+ \partial_r^2 A(x_0,r)+ \frac{m-1}{r} \partial_r A(x_0,r)&=&\partial_{x_0} \partial_r B(x_0,r)+ \frac{m-1}{r} \partial_{x_0}B(x_0,r)\\
&&- \partial_r \partial_{x_0}B(x_0,r)- \frac{m-1}{r} \partial_{x_0} B(x_0,r)\\
&=&0.
\end{eqnarray*}

By using similar arguments, \eqref{h12} follows.
\end{proof}

By means of the above result and Proposition \ref{split} it is possible to uncover other relations between the Fueter-Sce's theorem and the generalized CK-extension for axially harmonic functions.

\begin{prop}
Let $f(u+iv)= \alpha(u,v)+i \beta(u,v)$ be an intrinsic holomorphic function defined on an intrinsic complex domain $\Omega_2 \subset \mathbb{C}$. Then for $m \geq 3$ and odd we have
\begin{equation}
\label{s1}
[\Delta^{\frac{m-1}{2}}f]_1= \frac{\gamma_m}{m} HGCK[0, f^{(m)}(x_0)],
\end{equation}
and
\begin{equation}
\label{s2}
[\Delta^{\frac{m-1}{2}}f]_0= \gamma_m HGCK[ f^{(m-1)}(x_0),0 ],
\end{equation}
Setting $\Omega_1= \Omega_2 \cap \mathbb{R}$, and $ \Omega$ being an axially symmetric set, we have the following diagrams
\begin{equation}
\label{Diag3}
	\begin{tikzcd}[row sep = 3em, column sep = 6em]
		\mathcal{A}(\Omega_1)\otimes \mathbb{R}_m\arrow[r, "S", rightarrow] \arrow[d, "\frac{\gamma_m}{m} (0 {,} \partial_{x_0}^{m})", labels=left] & \mathcal{SM}(\Omega) \arrow[d, "\Delta_{\mathbb{R}^{m+1}}^{\frac{m-1}{2}} \circ ()_1"] \\
		(\mathcal{A}(\Omega_1)\otimes \mathbb{R}_m)^2 \arrow[r, "HGCK", rightarrow]    \arrow[d, "m P_2", labels=left] & \mathcal{AH}(\Omega) \arrow[d, "\overline{\mathcal{D}}"]\\
		\mathcal{A}(\Omega_1)\otimes \mathbb{R}_m \arrow[r, " GCK", rightarrow]                 & \mathcal{AM}(\Omega)\\
	\end{tikzcd} 
\vspace*{-4mm}
\begin{tikzcd}[row sep = 3em, column sep = 6em]
	\mathcal{A}(\Omega_1)\otimes \mathbb{R}_m\arrow[r, "S", rightarrow] \arrow[d, "\gamma_m (\partial_{x_0}^{m-1} {,} 0)", labels=left] & \mathcal{SM}(\Omega) \arrow[d, "\Delta_{\mathbb{R}^{m+1}}^{\frac{m-1}{2}} \circ ()_0"] \\
	(\mathcal{A}(\Omega_1)\otimes \mathbb{R}_m)^2 \arrow[r, "HGCK", rightarrow]    \arrow[d, "P_1", labels=left] & \mathcal{AH}(\Omega) \arrow[d, "\overline{\mathcal{D}}"]\\
	\mathcal{A}(\Omega_1)\otimes \mathbb{R}_m \arrow[r, " \partial_{x_0}\circ GCK ", rightarrow]                 & \mathcal{AM}(\Omega)\\
\end{tikzcd}
\end{equation}
where $P_1$ and $P_2$ are projections.
\end{prop}
\begin{proof}
We start by proving \eqref{s1}. By Proposition \ref{split} and Theorem \ref{FG} we have
$$[\Delta^{\frac{m-1}{2}}_{\mathbb{R}^{m+1}}f]_1=\gamma_m [GCK[f^{(m-1)}(x_0)]]_1=\frac{\gamma_m}{m}HGCK[0, f^{(m)}(x_0)]. $$

We show \eqref{s2}. By using another time Proposition \ref{split} and Theorem \ref{FG} we have
$$[\Delta^{\frac{m-1}{2}}_{\mathbb{R}^{m+1}}f]_0=\gamma_m [GCK[f^{(m-1)}]]_0= \gamma_m HGCK[f^{(m-1)}(x_0),0].$$
\end{proof}

We can unify \eqref{Diag2} and \eqref{Diag3} in the following diagram

\begin{figure}[H]
	\centering
	\resizebox{0.95\textwidth}{!}{%
		\input{Diagtotal.tikz}
	}
\end{figure}

Finally, we get the following plane wave decomposition of the factorization of the Fueter-Sce map.

\begin{prop}
Let $\Omega_2 \subset \mathbb{C}$ be an intrinsic complex domain and let $f: \Omega_2 \to \mathbb{C}$ be a holomorphic intrinsic function. Then for all $m \in \mathbb{N}$ being odd we have
$$\Delta_{\mathbb{R}{^{m+1}}}^{\frac{m-3}{2}} \mathcal{D}f(x_0+\underline{x})=  \frac{\gamma_m \Gamma \left(\frac{m}{2}\right)}{2 \pi^{\frac{m}{2}}} \int_{\mathbb{S}^{m-1}}[\cosh(\langle \underline{x}, \underline{\omega} \rangle) \underline{\omega} \partial_{x_0}] f^{(m-2)}(x_0) dS_{\underline{\omega}},$$
where the constant $\gamma_m$ is as in Theorem \ref{FG}.
\end{prop}
\begin{proof}
The result follows by combining Theorem \ref{Sceconn} and Theorem \ref{plane}.
\end{proof}
\section{A sequence of harmonic polynomials related to the Fueter-Sce's theorem}	

In \cite{DDG, DDG1}, the authors use the connection between the generalized CK extension for axially monogenic functions and the Fueter-Sce's theorem to show the following relation
$$ \Delta^{\frac{m-1}{2}}_{\mathbb{R}^{m+1}}(x^{m+k-1})= \gamma_m \frac{(m+k-1)!}{k!} \mathcal{Q}_k^m(x),$$
where $\gamma_m$ is the constant defined in Theorem \ref{FG} and $\mathcal{Q}_k^m(x)$ are the Clifford-Appell polynomials, see \cite{CFM, CMF}. These polynomials are axially monogenic and are defined as
\begin{equation}
\label{CliffAppe}
	\mathcal{Q}_{k}^{m}(x):= \sum_{s=0}^{k}\mathcal{T}_s^k(m)x^{k-s}\overline{x}^{s}, \qquad x \in \mathbb{R}^{m+1},
\end{equation}
where the coefficients $\mathcal{T}_s^k(m)$ are given by
\begin{equation}
	\label{f0}
	\mathcal{T}_s^k(m):= \binom{k}{s}\frac{ \left( \frac{m+1}{2}\right)_{k-s} \left( \frac{m-1}{2}\right)_s}{(m)_k}, \qquad m \geq 1 ,
\end{equation}
and $(.)_k$ stands for the Pochammer symbol, i.e. $(a)_s= \frac{\Gamma(a+s)}{\Gamma(a)}$. We observe that the Clifford-Appell polynomials are the generalized CK-extension of the monomial $x_0^k$, i.e.
\begin{equation}
\label{monoGCK}
\mathcal{Q}_k^m(x)=GCK[x_0^k].
\end{equation}
It is worth mentioning that in \cite{AKS3}, the authors established that Clifford-Appell polynomials serve as a basis for the space of axially monogenic functions. Moreover, in \cite{ACDDS}, these polynomials were utilized to extend Schur analysis techniques for axially monogenic functions. They were further applied in \cite{DDG1} to generalize various operators for Hardy and Fock spaces in the framework of axially monogenic functions

\medskip

In this section, we describe a possible application of Theorem \ref{Sceconn}. This allows to make a relation between the factorization of the Fueter-Sce map $ \Delta_{\mathbb{R}^{m+1}}^{\frac{m-1}{2}} \mathcal{D}$ and the following new class of polynomials. 
\begin{definition}
\label{polyH}
Let $k$, $m \in \mathbb{N}$ and $x \in \mathbb{R}^{m+1}$ then we define
$$ \mathcal{P}_k^m(x):= \sum_{s=0}^k T_{s}^k(m) x^{k-s} \bar{x}^s,$$
where $T_s^k(m):= \binom{k}{s} \frac{\left(\frac{m-1}{2}\right)_{k-s} \left(\frac{m-1}{2}\right)_{s}}{(m-1)_k}.$
\end{definition}

Now, our goal is to show that the polynomials introduced above are harmonic. To do this we need some preliminaries results. First we need the following result that puts in relation the coefficients of the Clifford-Appell polynomials and the polynomials $\mathcal{P}_k^m(x)$.

\begin{lemma}
Let $k \in \mathbb{N}$ and $0 \leq s \leq k-1$. Then we have
\begin{equation}
\label{idd}
(m-1)T_k^k(m)= (k+m-1) \mathcal{T}_k^k(m),
\end{equation}
and
\begin{equation}
\label{idd1}
(m-1)T_s^k(m)= (k+m-1) \mathcal{T}_s^k(m)-k \mathcal{T}_s^{k-1}(m).
\end{equation}
\end{lemma}
\begin{proof}
To show \eqref{idd} and \eqref{idd1} we use the following facts on the pochammer symbol:
\begin{equation}
\label{prel}
(a)_\ell=(\ell+a-1)(a)_{\ell-1}, \quad \frac{(a-1)}{(a-1)_\ell}= \frac{(a+\ell-1)}{(a)_\ell}, \quad a_{\ell}-\ell (a)_{\ell-1}=(a-1)_{\ell}, \qquad a, \ell \in \mathbb{N}. 
\end{equation}
We start proving \eqref{idd}. By the definitions of the coefficients $\mathcal{T}_s^k(m)$ and $T_s^k(m)$, and the second relation in \eqref{prel} we have
$$ (m-1)T_k^k(m)=\frac{(m-1)}{(m-1)_k}\left(\frac{m-1}{2}\right)_{k}=\frac{(k+m-1)}{(m)_k}\left(\frac{m-1}{2}\right)_{k}=(k+m-1) \mathcal{T}_k^k(m).$$

By \eqref{prel}, and the definitions of $\mathcal{T}_s^k(m)$ and $T_s^k(m)$ we have
\begingroup\allowdisplaybreaks
\begin{eqnarray*}
(k+m-1) \mathcal{T}_s^k(m)-k \mathcal{T}_s^{k-1}(m)&=&(k+m-1) \binom{k}{s} \frac{\left(\frac{m+1}{2}\right)_{k-s}\left(\frac{m-1}{2}\right)_{s}}{(m)_k}\\
&&-k \binom{k-1}{s} \frac{\left(\frac{m+1}{2}\right)_{k-1-s}\left(\frac{m-1}{2}\right)_{s}}{(m)_{k-1}}\\
&=& \frac{(k+m-1)}{(m)_k}\left[ \binom{k}{s} \left( \frac{m+1}{2}\right)_{k-s}-k \binom{k-1}{s} \left( \frac{m+1}{2}\right)_{k-s-1}\right] \cdot\\
&&\cdot  \left( \frac{m-1}{2}\right)_{s}\\
&=& \binom{k}{s} \frac{(k+m-1)}{(m)_k}\left[  \left( \frac{m+1}{2}\right)_{k-s}-(k-s) \left( \frac{m+1}{2}\right)_{k-s-1}\right] \cdot\\
&&\cdot  \left( \frac{m-1}{2}\right)_{s}\\
&=& \binom{k}{s} \frac{(k+m-1)}{(m)_k} \left( \frac{m-1}{2}\right)_{k-s}\left( \frac{m-1}{2}\right)_{s}\\
&=&  \binom{k}{s} \frac{(m-1)}{(m-1)_k} \left( \frac{m-1}{2}\right)_{k-s} \left( \frac{m-1}{2}\right)_{s}\\
&=& (m-1)T_s^k(m).
\end{eqnarray*}
\endgroup
\end{proof}

From the above result we get the following relation between the polynomials $\mathcal{P}_k^m(x)$ and the Clifford-Appell polynomials $Q_k^m(x)$.

\begin{prop}
\label{relll}
Let $m$, $k \in \mathbb{N}$ then we have
\begin{equation}
\label{rell}
(m-1) \mathcal{P}_k^m(x)=(k+m-1) \mathcal{Q}_{k}^m(x)- kx \mathcal{Q}_{k-1}^m(x), \qquad x \in \mathbb{R}^{n+1}.
\end{equation}
\end{prop}
\begin{proof}
By Definition \ref{polyH}, \eqref{idd} and \eqref{idd1} we have
\begin{eqnarray*}
(m-1) \mathcal{P}_k^m(x)&=& (m-1) \sum_{s=0}^{k} T_s^k(m) x^{k-s} \bar{x}^s\\
&=&(m-1) \sum_{s=0}^{k-1} T_s^k(m) x^{k-s} \bar{x}^s+(m-1)T_k^k(m)\bar{x}^k\\
&=& (k+m-1) \sum_{s=0}^{k-1} \mathcal{T}_s^k(m) x^{k-s} \bar{x}^s-k \sum_{s=0}^{k-1}\mathcal{T}_s^{k-1}(m) x^{k-s} \bar{x}^s\\
&&+(k+m-1) \mathcal{T}_k^{k}(m) \bar{x}^k\\
&=& (k+m-1) \sum_{s=0}^{k} \mathcal{T}_s^k(m) x^{k-s} \bar{x}^s-kx \sum_{s=0}^{k-1}\mathcal{T}_s^{k-1}(m) x^{k-s-1} \bar{x}^s\\
&=& (k+m-1) \mathcal{Q}_{k}^m(x)- kx \mathcal{Q}_{k-1}^m(x).
\end{eqnarray*}

\end{proof}

\begin{thm}
	Let $k$, $m \in \mathbb{N}$ and $x \in \mathbb{R}^{m+1}$. Then the polynomials $  \mathcal{P}_k^m(x)$ are axially harmonic.
\end{thm}
\begin{proof}
By the fact that the Clifford-Appell polynomials are of axial type and by \eqref{rell} we get that also the polynomials  $  \mathcal{P}_k^m(x)$ are of axial type. Now, we have to show that  $  \mathcal{P}_k^m(x)$ are harmonic. 
By the following well-known identity $\Delta_{\mathbb{R}^{m+1}}(xf(x))=x\Delta_{\mathbb{R}^{m+1}} f(x)+2 \mathcal{D}f(x)$, For any function $f$ that possesses continuous second-order derivatives, Proposition \ref{relll} and using the fact that the Clifford-Appell polynomials are monogenic (and so in particular harmonic) we have
\begin{eqnarray*}
\Delta_{\mathbb{R}^{m+1}} \mathcal{P}_k^m(x)&=& \frac{k+m-1}{m-1} \Delta_{\mathbb{R}{^{m+1}}} \mathcal{Q}_{k}^m(x)- \frac{k}{m-1}\Delta_{\mathbb{R}{^{m+1}}} \left(x \mathcal{Q}_{k-1}^m(x)\right)\\
&=& \frac{k+m-1}{m-1} \Delta_{\mathbb{R}{^{m+1}}} \mathcal{Q}_{k}^m(x)- \frac{k}{m-1}\Delta_{\mathbb{R}{^{m+1}}} \mathcal{Q}_{k-1}^m(x)- \frac{2k}{m-1} \mathcal{D}\mathcal{Q}_{k-1}^m(x)\\
&=&0.
\end{eqnarray*}

\end{proof}
Now, we establish a strong connection between the polynomials \( \mathcal{P}_k^m(x) \) and the following function:

\begin{equation}
	\label{rr}
	p(x) = |1-x|^{-(m-1)}, \qquad m \in \mathbb{N}, \qquad x \in \mathbb{R}^{m+1}.
\end{equation}

This function is widely known in the literature as the "Riesz potential", see \cite{M}.

\begin{thm}
	\label{expp}
Let $k$, $m \in \mathbb{N}$ and $x \in \mathbb{R}^{m+1}$. Then the Riesz potential admits the following expansion in series
$$ p(x)=\sum_{k=0}^\infty \frac{(m-1)_k}{k!}  \mathcal{P}_k^m(x), \qquad |x| <1.$$
\end{thm}
\begin{proof}
We start observing that we can write the function $p(x)$ as
$$ p(x)= (1-x)^{- \frac{m-1}{2}} (1- \bar{x})^{- \frac{m-1}{2}}.$$
If $|x| <1$, from the following well-known equality 
$$ (1-t)^{-s}= \sum_{j=0}^{\infty} \frac{(s)_j}{j!}  t^{j}, \qquad |t| <1,$$
we get
$$ p(x)= \left(\sum_{k=0}^\infty \frac{\left(\frac{m-1}{2}\right)_k}{k!} \bar{x}^k\right) \left(\sum_{k=0}^\infty \frac{\left(\frac{m-1}{2}\right)_k}{k!} x^k\right).$$
By the Cauchy product of series we obtain
\begin{eqnarray}
	\nonumber
	p(x) &=& \sum_{k=0}^\infty \sum_{s=0}^k \frac{\left(\frac{m-1}{2}\right)_{k-s}}{(k-s)!} \frac{\left(\frac{m-1}{2}\right)_s}{s!} x^{k-s} \bar{x}^{s}\\
	\nonumber
	&=& \sum_{k=0}^\infty \frac{(m-1)_k}{k!} \left(\sum_{s=0}^k \binom{k}{s} \frac{\left(\frac{m-1}{2}\right)_{k-s}\left(\frac{m-1}{2}\right)_s}{(m-1)_k} x^{k-s} \bar{x}^s\right)\\
	\label{f1}
	&=& \sum_{k=0}^\infty \frac{(m-1)_k}{k!}  \mathcal{P}_k^m(x).
\end{eqnarray}
This proves the result.
\end{proof}

Now, we demonstrate that the polynomials \( \mathcal{P}_k^m(x) \) have a deep connection with the Gegenbauer polynomials, which are defined as:

$$ C_k^{(\mu)}(x)= \sum_{\ell=\left[\frac{k}{2}\right]}^{k} \binom{-k}{\ell} \binom{\ell}{2\ell-k} (-2x)^{2\ell-k}, \quad x \in \mathbb{R}^{n+1},$$
with $k \in \mathbb{N}$ and $\mu >-1$, see \cite{S39}.
\begin{prop}
Let $x \in \mathbb{R}^{m+1}$. Then for $k \in \mathbb{N}$ we have
$$ \mathcal{P}_k^m(x)= \frac{k!}{(m-1)_k}|x|^k C_k^{\frac{m-1}{2}} \left(\frac{x}{|x|}\right), \quad |x|<1.$$
\end{prop}
\begin{proof}
	By \cite{GHS} we deduce that the function in \eqref{rr} has the following expansion in series
\begin{equation}
\label{exx}
p(x)= \sum_{k=0}^{\infty}C_{k}^{(\frac{m-1}{2})}\left(\frac{x}{|x|}\right)|x|^k, \qquad |x| <1.
\end{equation}
By equating \eqref{exx} with the expansion in series of the Riesz potential $p(x)$ in terms of the polynomials $ \mathcal{P}_k^m(x)$, see Theorem \ref{expp}, we get the result. 

\end{proof}

The coefficients of the polynomials $ \mathcal{P}_k^m(x)$ exhibit specific properties, as outlined in the following result, which will be useful in the subsequent sections.

\begin{lemma}
	\label{sum}
Let us consider $k \in \mathbb{N}$ and $m \geq 3$. The coefficients of the polynomials $ \mathcal{P}_k^m(x)$ satisfy the following properties:
	\begin{itemize}
		\item[1)] $T_{s}^k(m)=T_{k-s}^k(m), \qquad 0 \leq s \leq k,$
		\item[2)]  $ \sum_{s=0}^k T_s^k(m)=1.$
	\end{itemize}  
\end{lemma}
\begin{proof}
	\begin{itemize}
		\item[1)] Form the definition of pochhammer symbol can write the coeffients of the polynomial $  \mathcal{P}_k^m(x)$ as
		\begin{equation}
			\label{pochh}
			T_s^k(m) =\binom{k}{s} \frac{\Gamma(m-1)}{\Gamma(m-1+k)} \frac{\Gamma\left( \frac{m-1}{2}+k-s\right) \Gamma \left( \frac{m-1}{2}+s\right)}{\left[ \Gamma \left( \frac{m-1}{2}\right)\right]^2}.
		\end{equation}
		Therefore, we have
$$
			T_{k-s}^k(m)= \binom{k}{k-s} \frac{\Gamma(m-1)}{\Gamma(m-1+k)} \frac{ \Gamma \left( \frac{m-1}{2}+s\right)\Gamma\left( \frac{m-1}{2}+k-s\right)}{\left[ \Gamma \left( \frac{m-1}{2}\right)\right]^2}= T_s^k(m).
$$
		\item[2)] By formula \eqref{pochh} we have
		$$ \sum_{s=0}^k T_s^k(m)=\frac{\Gamma(m-1)}{\Gamma(m-1+k)\left[ \Gamma \left( \frac{m-1}{2}\right)\right]^2} \sum_{s=0}^{k} \binom{k}{s} \Gamma\left( \frac{m-1}{2}+k-s\right) \Gamma \left( \frac{m-1}{2}+s\right).$$
		Now, by using the following equality
		$$ \sum_{s=0}^{k} \binom{k}{s} \Gamma( \ell+1+k-s) \Gamma( \ell+1+s)= \frac{\Gamma(2 \ell+k+2)}{(\ell+1) \binom{2 \ell +1}{\ell+1}}, \qquad \ell \geq 1,$$
with $ \ell= \frac{m-3}{2}$, we obtain
		\begin{eqnarray*}
			\sum_{s=0}^k T_s^k(m) &=& \frac{\Gamma(m-1)}{\Gamma(m-1+k)\left[ \Gamma \left( \frac{m-1}{2}\right)\right]^2} \frac{\Gamma(m-1+k)}{ \left( \frac{m-1}{2}\right) \binom{m-2}{\frac{m-1}{2}}}\\
			&=& \frac{\Gamma(m-1)}{\Gamma(m-1+k)\left[ \Gamma \left( \frac{m-1}{2}\right)\right]^2}  \frac{\Gamma(m-1+k)}{ \left( \frac{m-1}{2}\right) \frac{\Gamma(m-1)}{\Gamma\left(\frac{m-1}{2}\right) \Gamma \left( \frac{m+1}{2}\right)}}\\
			&=& \frac{\Gamma \left( \frac{m+1}{2}\right)}{ \left(\frac{m-1}{2}\right)  \Gamma \left( \frac{m-1}{2}\right)}\\
			&=&1.
		\end{eqnarray*}
	\end{itemize}
\end{proof}

Now, we have all the tools to figure out the application of the operator $ \Delta_{\mathbb{R}^{m+1}}^{\frac{m-3}{2}} \mathcal{D}$ to a monomial $x^k$

\begin{thm}
\label{polyn}
Let $m \geq 3$, being odd, and $k \geq m-2$ then we have
\begin{eqnarray}
\label{t1}
\Delta_{\mathbb{R}^{m+1}}^{\frac{m-3}{2}} 	\mathcal{D} x^k&=&\gamma_{m} \frac{k!}{(k-m+2)!} HGCK(x_0^{k-m+2},0)\\
\label{t2}
&=&  \gamma_{m} \frac{k!}{(k-m+2)!} \mathcal{P}_{k-m+2}^m(x).
\end{eqnarray}
\end{thm}
\begin{proof}
Formula \eqref{t1} follows by applying Theorem \ref{Sceconn} and the identity $ \partial_{x_0}^{m-2}x_0^k= \frac{(k-m+2)!}{k!} x_0^{k-m+2}$. Now in order to have the second equality we have to prove that
$$HGCK(x_0^{k-m+2},0)=\mathcal{P}_{k-m+2}^m(x).$$
We have to look for the initial functions $A_0(x_0)$ and $A_1(x_0)$ for the harmonic generalized CK extension, see Theorem \ref{GCK}. The computations of the first initial function follows from the properties of the coefficients of $ \mathcal{P}_k^m(x)$. By the second point of Lemma \ref{sum} we have
$$ \mathcal{P}_{k-m+2}(x_0,0)=\sum_{j=0}^{k-m+2} T_j^{k-m+2}(m) x_{0}^{k-m+2}=x_0^{k-m+2}.$$
To compute the other initial function $A_1(x_0)$ we need to apply the operator $\partial_{\underline{x}}$ to the polynomials $ \mathcal{P}_{k-m+2}^m(x)$. In order to do this we divide the polynomials $ \mathcal{P}_k^m(x)$ in two pieces: 
\begin{eqnarray}
\nonumber
 \mathcal{P}_k^m(x)&=& \sum_{s=0}^{\floor*{\frac{k}{2}}} T_s^k(m) x^{k-s} \bar{x}^{s}+ \sum_{s= \ceil*{\frac{k}{2}}}^k T_s^k(m) x^{k-s} \bar{x}^s\\
\nonumber
&=& \sum_{s=0}^{\floor*{\frac{k}{2}}}T_s^k(m) x^{k-2s}x^s\bar{x}^{s}+ \sum_{s= \ceil*{\frac{k}{2}}}^k T_s^k(m) x^{k-s} \bar{x}^{k-s} \bar{x}^{2s-k}\\
\label{f2}
&=& \sum_{s=0}^{\floor*{\frac{k}{2}}} T_s^k(m) x^{k-2s}|x|^{2s}+ \sum_{s=\ceil*{\frac{k}{2}}}^k T_s^k(m) |x|^{2(k-s)} \bar{x}^{2s-k}.
\end{eqnarray}
By the binomial theorem we have
\begin{equation}
\label{f3}
 x^{k-2s}=(x_0+ \underline{x})^{k-2s}= \sum_{j=0}^{k-2s} \binom{k-2s}{j} \underline{x}^{j} x_0^{k-2s-j}, \quad
\end{equation}
\begin{equation}
\label{f4}
\bar{x}^{2s-k}=(x_0- \underline{x})^{2s-k}= \sum_{j=0}^{2s-k} \binom{2s-k}{j} (-\underline{x})^j x_0^{2s-k-j}.
\end{equation}
By the fact that
$$ \partial_{\underline{x}} \underline{x}^\ell=c(\ell,m) \underline{x}^{\ell-1}, \qquad
\begin{cases}
-2\ell  \qquad \ell \, \, \hbox{even}\\
-(2\ell+m) \qquad \ell \, \, \hbox{odd}.
\end{cases}
$$
and \eqref{f3} and \eqref{f4} we have
\begin{eqnarray}
\label{fr}
\partial_{\underline{x}} \mathcal{P}_k^m(x)&=& \sum_{s=0}^{\floor*{\frac{k}{2}}}T_s^k(m) \left( \sum_{j=1}^{k-2s} \binom{k-2s}{j}c(j,m) \underline{x}^{j-1} x_0^{k-2s-j}\right)(x_0^2+|\underline{x}|^2)^s\\
\nonumber
&&+\sum_{s=0}^{\floor*{\frac{k}{2}}}T_s^k(m) \left( \sum_{j=0}^{k-2s} \binom{k-2s}{j} \underline{x}^{j} x_0^{k-2s-j}\right) \left(\sum_{\alpha=1}^{s} \binom{s}{\alpha} x_0^{2s-2\alpha} c(2\alpha, m) (-\underline{x})^{2\alpha-1} \right)\\
\nonumber
&&+ \sum_{s=\ceil*{\frac{k}{2}}}^k T_s^k(m) \left( \sum_{j=1}^{2s-k} \binom{2s-k}{j} c(j,m) (-1)^j \underline{x}^{j-1} x_0^{2s-k-j}\right)(x_0^2+|\underline{x}|^{2})^{k-s}\\
\nonumber
&&+\sum_{s=\ceil*{\frac{k}{2}}}^k T_s^k(m) \left( \sum_{j=1}^{2s-k} \binom{2s-k}{j} (-\underline{x})^{j} x_0^{2s-k-j}\right)\cdot\\
\nonumber
&& \, \, \, \, \, \, \cdot \left( \sum_{\beta=1}^{k-s} \binom{k-s}{\beta} x_0^{2(k-s-\beta)} c(2 \beta,m) (-\underline{x})^{2\beta-1}\right).
\end{eqnarray}
Thus by making the restriction of \eqref{fr} at $\underline{x}=0$ we get
\begin{eqnarray*}
\partial_{\underline{x}}[\mathcal{P}_k^m(x)] \biggl |_{\underline{x}=0}&=& c(1,m) \sum_{s=0}^{\floor*{\frac{k}{2}}}T_s^k(m) (k-2s)  x_0^{k-2s-1}x_0^{2s}-c(1,m)\sum_{s=\ceil*{\frac{k}{2}}}^k T_s^k(m) (2s-k)  x_0^{2s-k-1}x_0^{2k-2s}\\
&=& -(m+2)\left(\sum_{s=0}^{\floor*{\frac{k}{2}}}T_j^k(m) x_{0}^{k-1} (k-2s) + \sum_{s=\ceil*{\frac{k}{2}}}^{k}T_j^k(m) x_{0}^{k-1}(k-2s)\right)\\
&=&-(m+2) x_{0}^{k-1} \sum_{j=0}^{k} T_j^k(m) (k-2j).
\end{eqnarray*}
 Therefore by \eqref{ini} we have that
$$ A_{1}(x_0)= -\frac{1}{m}\partial_{\underline{x}}[\mathcal{P}_k^m(x)] \biggl |_{\underline{x}=0}=- \frac{m+2}{m}x_{0}^{k-1} \sum_{j=0}^k T_j^k(m) (k-2j).$$
Finally, by the first point of Lemma \ref{sum} we get
\begin{eqnarray*}
A_{1}(x_0)&=& - \frac{m+2}{m}x_{0}^{k-1}\sum_{j=0}^k T_j^k(m) (k-2j)\\
&=& - \frac{m+2}{m}x_{0}^{k-1} \bigl[(T_0^k(m)k-T_k^k(m) k)+ [T_1^k(m)(k-2)-T_{k-1}^k(m)(k-2)]+ \\
&& ...+(2T_{\floor{\frac{k}{2}-1}}^k(m)-2 T_{\ceil{\frac{k}{2}+1}}^k(m)) +T_{\left[\frac{k}{2}\right]}^k(m) \cdot 0 \biggl]\\
&=&0.
\end{eqnarray*}
This proves \eqref{t2}.
\end{proof}
Now, we show that the polynomials $ \mathcal{P}_k^m(x)$ can be  written in terms of integrals over the sphere $ \mathbb{S}^{m-1}$ involving functions of plane wave type. 
\begin{prop}
Let $k \in \mathbb{N}$ and $m \in \mathbb{N}$ being an odd number. Then we have
$$ \mathcal{P}_k^m(x)= HGCK[x_0^k,0]= \frac{\Gamma \left(\frac{m}{2}\right)}{2 \pi^{\frac{m}{2}}}  \int_{\mathbb{S}^{m-1}}\cosh\left(\langle \underline{x}, \underline{\omega} \rangle \underline{\omega} \partial_{x_0}\right) x_0^k dS_{\underline{\omega}}.$$
\end{prop}
\begin{proof}
The result follows by combining Theorem \ref{polyn} and Theorem \ref{plane}.
\end{proof}
Now we give the following important property for the harmonic polynomials $ \mathcal{P}_k^m(x)$.
\begin{prop}
Let $m\geq 3$ being odd, $ k \geq m-2$ and $x \in \mathbb{R}^{n+1}$. Then the polynomials $ \mathcal{P}_{k-m+2}^m(x)$ are real-valued.
\end{prop}
\begin{proof}
Let $ f(x_0+ \underline{x})=\alpha(x_0, r)+ \underline{\omega} \beta(x_0, r)$ be a slice monogenic function. By Theorem \ref{rapp} we know that $\alpha(x_0, \underline{x})$ and $\beta(x_0, \underline{x})$ satisfy the Cauchy-Riemann conditions, see \eqref{R1}.
Hence by Lemma \ref{Newt} we get that $ \mathcal{D} f(x_0+ \underline{x})$ is real. 
Now, since the monomial $x^k$ is slice monogenic for every $k \geq 0$ then $ \mathcal{D} x^k$ is real. By the fact that the Laplacian is a real-valued operator, the polynomials generated by $\Delta_{\mathbb{R}^{m+1}}^{\frac{m-3}{2}} \mathcal{D} x^k$ are real-valued. Therefore by Theorem \ref{polyn} we get that the polynomials $\mathcal{P}_{k-m+2}(x)$ are real-valued, too.
\end{proof}

\begin{rem}
The result in Theorem \ref{polyn} extends to arbitrary dimension the one obtained for the quaternionic setting in \cite{B}. Indeed, if we consider $m=3$ in \eqref{t2}, by the definition of the polynomials $ \mathcal{P}_k^m(x)$ we get
\begin{eqnarray*}
\mathcal{D} x^k &=& -\frac{2 k!}{(k-1)!} \sum_{s=0}^{k-1} \frac{(k-1)!}{s! (k-1-s)!} \frac{\Gamma(2)}{\Gamma(k+1)} \frac{\Gamma(k-s)}{[\Gamma(1)]^2} \Gamma(1+s) x^{k-1-s} \bar{x}^{s}\\
&=&-2 \sum_{s=0}^{k-1} x^{k-1-s} \bar{x}^s= -2 \sum_{s=1}^k x^{k-s} \bar{x}^{s-1},
\end{eqnarray*}
which is exactly the same result obtained in \cite[Lemma 1]{B}.
\end{rem}

Now, we want to determine a basis of the set of axially harmonic functions. In order to do this we need to figure out the behaviour of the harmonic  generalized CK extension applied to the basic monomial $x_0^k$, building block of the Taylor expansion of a real analytic function.

\begin{thm}
\label{basis}
The space of axially harmonic functions is generated by polynomials of the form
\begin{equation}
\label{f6}
H_k^0(x)= HGCK[x_0^k,0] \qquad H_k^1(x)=HGCK[0, x_0^k], \quad x \in \mathbb{R}^{m+1}
\end{equation}
Moreover, one has that
$$ H_k^0(x)=[\mathcal{Q}_k^m(x)]_0 \qquad \hbox{and} \qquad H_k^1(x)=\frac{m}{k+1} [\mathcal{Q}_{k+1}^m(x)]_1,$$
where the polynomials $ \mathcal{Q}_k^m(x)$ are the Clifford-Appell polynomials, see \eqref{CliffAppe}.
\end{thm}

\begin{proof}

By Proposition \ref{split} and formula \eqref{monoGCK} we have
\begin{eqnarray*}
H_k^0(x)&=& HGCK[x_0^k,0]\\
&=&  \left[CGK[x_0^k]\right]_0\\
&=& [\mathcal{Q}_k^m(x)]_{0}.
\end{eqnarray*}

By using similar arguments we have
\begin{eqnarray*}
H_{k}^{1}(x)&=& HGCK[0, x_0^k]\\
&=& \frac{m}{k+1} \left[ CGK[x_{0}^{k+1}]\right]_1\\
&=& \frac{m}{k+1} [\mathcal{Q}_{k+1}^m(x)]_1.
\end{eqnarray*}
All the elements in the set of axially harmonic functions can be written as
$$ f(x)= HGCK[f_{0}(x_0), g_{0}(x_0)],$$
where $f$, $g$ are real analytic functions on the real line. Finally, by formula \eqref{f6} we have
\begin{eqnarray*}
	f(x) &=& HGCK \left[ \sum_{j=0}^\infty x_{0}^j a_{j}, \sum_{\ell=0}^\infty x_{0}^\ell  b_{\ell}\right]\\
	&=& \sum_{j=0}^\infty HGCK[x_{0}^j,0]a_{j}+ \sum_{\ell=0}^\infty HGCK[0, x_{0}^\ell] b_{\ell}\\
	&=& \sum_{j=0}^\infty H_{j}^0(x) a_j+ \sum_{\ell=0}^\infty H_{\ell}^1(x) b_{\ell},
\end{eqnarray*}
where $ \{a_j\}_{j \geq 0}$, $ \{b_{\ell}\}_{\ell \geq 0} \subseteq \mathbb{R}_n$.

\end{proof}

Now, we give an explicit expression of the polynomials $ H_k^0(x)$ and $H_k^1(x)$ in terms of $x$ and $\bar{x}$.
\begin{prop}
Let $k \geq 0$, $m \in \mathbb{N}$ being an odd number and $x \in \mathbb{R}^{n+1}$ then we have
\begin{equation}
\label{bb}
 H_{k}^{0}(x)=\mathcal{P}_k^m(x),
\end{equation}
and
$$ H_{k}^{1}(x)= \frac{m}{(k+1)(m+k)}\sum_{s=0}^{k+1}\binom{k+1}{s}\frac{ \left(\frac{m-1}{2}\right)_s\left( \frac{m-1}{2}\right)_{k+1-s}}{(m-1)_{k+1}} (k+1-2s) x^{k-s} \bar{x}^s.$$
\end{prop}	
\begin{proof}
By Theorem \ref{basis} and Theorem \ref{polyn} we have that
$$ H_0^k(x)=HGCK[x_0^k,0]= \mathcal{P}_k^m(x).$$
Now, by \eqref{bb} we have deduce that
$$ \mathcal{Q}_k^m(x)=[\mathcal{Q}_k^m(x)]_0+[\mathcal{Q}_k^m(x)]_1=\mathcal{P}_k^m(x)+[\mathcal{Q}_k^m(x)].$$
By Theorem \ref{basis} and using the definition of the Clifford-Appell polynomials and the definition of the polynomials $ \mathcal{P}_k^m(x)$ we have
\begin{equation}
\label{1vec}
H_k^1(x)= \frac{m}{k+1}[\mathcal{Q}_{k+1}^m(q)]_1=\frac{m}{k+1}\left[\mathcal{Q}_{k+1}^m(x)-\mathcal{P}_{k+1}^m(x)\right]=\frac{m}{k+1}\sum_{s=0}^{k+1}\left[\mathcal{T}_s^{k+1}(m)-T_s^{k+1}(m)\right] x^{k+1-s} \bar{x}^s.
\end{equation}
By the definition of the coefficients $\mathcal{T}_s^k(m)$, see \eqref{f0}, and $T_s^k(m)$, see Definition \ref{polyH}, and the following fact on the pochhammer symbol
$$ (a+1)_{\ell}= \frac{a+\ell}{a} (a)_{\ell}, \qquad a, \ell \in \mathbb{N},$$
we have
\begin{eqnarray}
\nonumber
\mathcal{T}_s^{k+1}(m)-T_s^{k+1}(m)&=& \binom{k+1}{s} \left(\frac{m-1}{2}\right)_s \left[ \frac{\left( \frac{m+1}{2}\right)_{k+1-s}}{(m)_{k+1}}- \frac{\left( \frac{m-1}{2}\right)_{k+1-s}}{(m-1)_{k+1}}\right]\\
\nonumber
&=& \binom{k+1}{s}\frac{ \left(\frac{m-1}{2}\right)_s\left( \frac{m-1}{2}\right)_{k+1-s}}{(m-1)_{k+1}}\left[ \frac{\left(\frac{m-1}{2}+k+1-s\right)(m-1)}{\left(\frac{m-1}{2}\right)(m+k)}-1 \right]\\
\label{1vecb}.
&=& \binom{k+1}{s}\frac{ \left(\frac{m-1}{2}\right)_s\left( \frac{m-1}{2}\right)_{k+1-s}}{(m-1)_{k+1}} (k+1-2s)
\end{eqnarray}
Hence the result follows by plugging \eqref{1vecb} into \eqref{1vec} we get the result.
\end{proof}

Now, we show that the polynomials $H_k^1(x)$ can be  written in terms of integrals over the sphere $ \mathbb{S}^{m-1}$ involving functions of plane wave type. 
\begin{prop}
	Let $k \in \mathbb{N}$ and $m \in \mathbb{N}$ being an odd number. Then we have
	$$ H_k^1(x)= \frac{\Gamma(m+1)}{\pi^{\frac{m}{2}}} \int_{\mathbb{S}^{m-1}} \sinh \left(\langle \underline{x}, \underline{\omega} \rangle \underline{\omega} \partial_{x_0}\right) x_0^k.$$
\end{prop}
\begin{proof}
	The result follows by combining Theorem \ref{basis} and Theorem \ref{plane}.
\end{proof}
	
The Clifford-Appell polynomials introduced in \eqref{CliffAppe} satisfy the so-called Appell property. In classic terms, a sequence of polynomials $ \{P_k\}_{k \in \mathbb{N}}$ of the real variable $x_0$ is called an Appell sequence if it satisfies the property
$$ \frac{d}{d x_0} P_k(x_0)=k P_{k-1}(x_0).$$
The classical monomials $ \{x_0^k\}_{k \in \mathbb{N}}$ as well as the Hermite, Bernoulli, and Euler polynomials are interesting examples of such polynomials. 
\\ In Clifford analysis, Appell sequences have been studied in relation to the action of the so-called hypercomplex derivative ( $ \mathcal{\overline{D}}:= \frac{1}{2}(\partial_{x_0}- \partial_{\underline{x}})$) on the Clifford-Appell polynomials, i.e. 
$$ \frac{1}{2} \overline{\mathcal{D}} \mathcal{Q}_k^m(x)=k \mathcal{Q}_{k-1}^m(x).$$	
This is a direct consequence of the monogenicity of the polynomials $ \mathcal{Q}_k^m(x)$. 
\\In the following result we express a sort of Appell properties for the polynomials $H_k^0(x)$  and $H_k^1(x)$. 	
	
\begin{thm}
Let $k \in \mathbb{N}$ and $x \in \mathbb{R}^{n+1}$. Then the polynomials $H_k^0(x)$ and $H_k^1(x)$ hold the following Appell-like properties
$$ \begin{cases}
\partial_{x_0} H_k^0(x)=- \langle \partial_{\underline{x}}, H_{k-1}^1(x) \rangle\\
\partial_{\underline{x}} H_k^0(x)= \frac{k}{m}\partial_{x_0} H_{k-1}^1(x)\\
\partial_{\underline{x}} \wedge H_{k-1}^1(x)=0.
\end{cases}
$$
\end{thm}
\begin{proof}
By Theorem \ref{basis} we deduce that we can write the Clifford-Appell polynomials in the following way
\begin{equation}
\label{Cliff}
\mathcal{Q}_{k}^m(x)= H_{k}^{0}(x)+\frac{k}{m}H_{k-1}^{1}(x).
\end{equation}
Now, we act with the Dirac operator on $ \mathcal{Q}_k^m(x)$ and we examine its $k$-vector structure:
\begin{equation}
\label{Cliff3}
\partial_{\underline{x}}\mathcal{Q}_k^m(x)=[\partial_{\underline{x}} \mathcal{Q}_{k}^m(x)]_0+[\partial_{\underline{x}} \mathcal{Q}_{k}^m(x)]_1+[\partial_{\underline{x}} \mathcal{Q}_{k}^m(x)]_2,
\end{equation}
where
$$ 
\begin{cases}
[\partial_{\underline{x}} \mathcal{Q}_k^m(x)]_0=- \langle \partial_{\underline{x}}, H_{k-1}^1(x) \rangle\\
[\partial_{\underline{x}} \mathcal{Q}_k^m(x)]_1=\partial_{\underline{x}} H_{k}^0(x)\\
[\partial_{\underline{x}} \mathcal{Q}_k^m(
x)]_2= \partial_{\underline{x}} \wedge H_{k-1}^1(x).
\end{cases}
$$
On the other hand, since the Clifford-Apell are monogenic and by \eqref{Cliff} we get
\begin{eqnarray}
\nonumber
\partial_{\underline{x}} \mathcal{Q}_k^m(x) 
\nonumber
&=& - \partial_{x_0} \mathcal{Q}_k^m(x)\\
\nonumber
&=& -k \mathcal{Q}_{k-1}^m(x)\\
\label{Cliff1}
&=& -kH_{k-1}^0(x)-\frac{k(k-1)}{m} H_{k-2}^1(x).
\end{eqnarray}
By making equal formula \eqref{Cliff3} and \eqref{Cliff1} we get
\begin{equation}
\label{Cliff8}
\begin{cases}
\langle \partial_{\underline{x}}, H_{k-1}^1(x) \rangle =k H_{k-1}^0(x)\\
\partial_{\underline{x}} H_{k}^0(x)=-\frac{k(k-1)}{m} H_{k-2}^1(x)\\
\partial_{\underline{x}} \wedge H_{k-1}^1(x)=0.
\end{cases}
\end{equation}
Now, we observe that
\begin{equation}
\label{Cliff2}
\partial_{x_0} \mathcal{Q}_k^m(x)= \partial_{x_0} H_{k}^0(x)+ \frac{k}{m} \partial_{x_0} H_{k-1}^1(x).
\end{equation}
At the same time we have
\begin{eqnarray}
\label{Cliff5}
\partial_{x_0} \mathcal{Q}_k^m(x) = - k H_{k-1}^0(x)-\frac{k(k-1)}{m} H_{k-2}^1(x).
\end{eqnarray}
Therefore by putting equal \eqref{Cliff2} and \eqref{Cliff5} we get
\begin{equation}
\label{Cliff6}
\partial_{x_0} H_{k}^0 (x)=-k H_{k-1}^0(x),
\end{equation}
and
\begin{equation}
\label{Cliff7}
 \partial_{x_0} H_{k-1}^1(x)=-(k-1) H_{k-2}^{1}(x).
\end{equation}
Finally, by substituting \eqref{Cliff6} and \eqref{Cliff7} in the system \eqref{Cliff8} we get the statement.
\end{proof}

\section{Concluding remarks}

In this paper, we have developed a generalized CK-extension for axially harmonic functions. This CK-extension involves two initial functions, upon which two power series of differential operators act. Additionally, we have established a connection between the factorization of the Fueter-Sce map and the generalized CK-extension for axially harmonic functions.

A natural extension of axially harmonic functions is represented by axially polyharmonic functions, which are axial functions in the kernel of the operator $ \Delta^{\ell}_{\mathbb{R}^{m+1}} $, with $ \ell \in \mathbb{N} $ (see \cite{Aro}). The relationship between this class of functions and the Fueter-Sce mapping theorem has already been explored in \cite{Fivedim}.

It is therefore natural to ask whether a generalized CK-extension for axially polyharmonic functions can be constructed and whether a deeper connection with the Fueter-Sce theorem can be established. This question will be the focus of a forthcoming paper.

\end{document}